\documentclass[a4paper,12pt]{article}
\usepackage{tikz}
\usetikzlibrary{matrix,arrows,decorations.markings}
\usepackage{amsmath,amsfonts,amssymb,amsthm,url,textcomp,mathrsfs}
\usepackage{mathtools}
\usepackage{csquotes}
\usepackage{a4wide}
\usepackage[numbers]{natbib}
\usepackage{fancyhdr}
\usepackage{anyfontsize}
\usepackage{hyperref}
\usepackage{hhline}
\usepackage{leftidx}
\usepackage{enumitem}
\usepackage{mdframed}
\usepackage[ruled,vlined]{algorithm2e}
\usepackage{makecell}
\usepackage{longtable}

\allowdisplaybreaks

\makeatletter
\let\@fnsymbol\@arabic
\makeatother

\newtheorem{theorem}{Theorem}\numberwithin{theorem}{section}
\newtheorem*{theorem*}{Theorem}
\newtheorem*{lemma*}{Lemma}

\newtheorem*{definition*}{Definition}
\newtheorem*{proposition*}{Proposition}
\newtheorem{lemma}[theorem]{Lemma}

\newtheorem{proposition}[theorem]{Proposition}

\newtheorem{theoremm}{Theorem}\numberwithin{theoremm}{subsection}
\newtheorem{deffinition}[theoremm]{Definition}
\newtheorem{lemmma}[theoremm]{Lemma}
\newtheorem{corrollary}[theoremm]{Corollary}
\newtheorem{nottation}[theoremm]{Notation}
\newtheorem{propposition}[theoremm]{Proposition}

\newtheorem{quesstion}[theoremm]{Question}
\newtheorem{problemm}[theoremm]{Problem}
\numberwithin{theoremmm}{subsubsection}

\theoremstyle{remark}

\newtheorem*{example*}{Example}

\newcommand{\Rad}{\operatorname{Rad}}
\newcommand{\Aut}{\operatorname{Aut}}
\newcommand{\Alt}{\operatorname{Alt}}

\newcommand{\ord}{\operatorname{ord}}
\newcommand{\Sym}{\operatorname{Sym}}

\newcommand{\id}{\operatorname{id}}

\newcommand{\e}{\mathrm{e}}
\newcommand{\M}{\operatorname{M}}

\newcommand{\GL}{\operatorname{GL}}

\newcommand{\D}{\operatorname{D}}

\newcommand{\Mod}[1]{\ (\textup{mod}\ #1)}

\newcommand{\IF}{\mathbb{F}}

\newcommand{\IZ}{\mathbb{Z}}

\newcommand{\Q}{\operatorname{Q}}

\newcommand{\inv}{\operatorname{inv}}

\newcommand{\Ccal}{\mathcal{C}}

\newcommand{\imp}{\mathrm{imp}}

\newcommand{\Bcal}{\mathcal{B}}

\newcommand{\AGL}{\operatorname{AGL}}

\newcommand{\SD}{\operatorname{SD}}
\newcommand{\SmallGroup}{\operatorname{SmallGroup}}
\newcommand{\AC}{\operatorname{AC}}
\newcommand{\comp}{\operatorname{comp}}
\newcommand{\orth}{\operatorname{orth}}

\newcommand{\Cscr}{\mathscr{C}}

\begin{document}

\title{Compositions and parities of complete mappings and of orthomorphisms}

\author{Alexander Bors\textsuperscript{1} \and Qiang Wang\thanks{School of Mathematics and Statistics, Carleton University, 1125 Colonel By Drive, Ottawa ON K1S 5B6, Canada. \newline First author's e-mail: \href{mailto:alexanderbors@cunet.carleton.ca}{alexanderbors@cunet.carleton.ca} \newline Second author's e-mail: \href{mailto:wang@math.carleton.ca}{wang@math.carleton.ca} \newline The authors are supported by the Natural Sciences and Engineering Research Council of Canada (RGPIN-2017-06410). \newline 2020 \emph{Mathematics Subject Classification}: Primary: 20D60. Secondary: 05B15, 20D10, 20D15, 94A60. \newline \emph{Keywords and phrases}: Alternating group; Complete mappings; Cryptography; Finite fields; Finite groups; Harmonious groups; Latin squares; Orthomorphisms; Permutation groups; Primitive groups; Round functions; Symmetric group.}}

\date{\today}

\maketitle

\abstract{We determine the permutation groups $P_{\comp}(\IF_q),P_{\orth}(\IF_q)\leq\Sym(\IF_q)$ generated by the complete mappings, respectively the orthomorphisms, of the finite field $\IF_q$ -- both are equal to $\Sym(\IF_q)$ unless $q\in\{2,3,4,5,8\}$. More generally, denote by $P_{\comp}(G)$, respectively $P_{\orth}(G)$, the subgroup of $\Sym(G)$ generated by the complete mappings, respectively the orthomorphisms, of the group $G$. Using recent results of Eberhard-Manners-Mrazovi{\'c} and M{\"u}yesser-Pokrovskiy, we show that for each large enough finite group $G$ that has a complete mapping (i.e., whose Sylow $2$-subgroups are trivial or noncyclic), $P_{\comp}(G)=\Sym(G)$ and $P_{\orth}(G)\geq\Alt(G)$. We also prove that $P_{\orth}(G)=\Sym(G)$ for every large enough finite \emph{solvable} group $G$ that has a complete mapping. Proving these results requires us to study the parities of complete mappings and of orthomorphisms. Some connections with known results in cryptography and with parity types of Latin squares are also discussed.}

\section{Introduction}\label{sec1}

\subsection{Background and main results}\label{subsec1P1}

Let $G$ be a group, written multiplicatively. A \emph{complete mapping of $G$} is a permutation $f$ of $G$ such that the function $\tilde{f}=\id\cdot f:G\rightarrow G, g\mapsto gf(g)$, is also a permutation of $G$. An \emph{orthomorphism of $G$} is a permutation $f$ of $G$ such that $g\mapsto g^{-1}f(g)$ is also a permutation of $G$. These two notions are closely linked -- in fact, there are two natural bijections from the set of complete mappings of $G$ to the set of orthomorphisms of $G$: the mapping $f\mapsto\tilde{f}$ (whence $\tilde{f}$ is also called the \emph{orthomorphism of $G$ associated with $f$}), and $f\mapsto f\circ\inv$, where $\inv:G\rightarrow G, g\mapsto g^{-1}$, is the inversion function of $G$. In case $G$ is abelian, the function $f\mapsto\inv\circ f$ is also such a bijection (and $\inv\circ f$ is more commonly written $-f$ if $G$ is written additively). See Evans' book \cite{Eva92a} for a concise introduction to the theory of complete mappings and orthomorphisms.

Complete mappings were originally introduced by Mann in 1942 \cite{Man42a}, motivated by a combinatorial application (the construction of mutually orthogonal Latin squares -- see also our Subsection \ref{subsec5P2}). Later, various authors studied the question which groups have at least one complete mapping. In 1950, Bateman \cite{Bat50a} proved that every infinite group has a complete mapping. For finite groups, the question turned out to be more delicate and led to the celebrated \emph{Hall-Paige Conjecture}, which is now a theorem thanks to the work of Hall and Paige \cite{HP55a}, Wilcox \cite{Wil09a}, Evans \cite{Eva09a}, and Bray et al. \cite[Section 2]{BCCSZ20a}. It states that a finite group $G$ has a complete mapping if and only if it satisfies what we will henceforth refer to as the \emph{Hall-Paige condition}: that the Sylow $2$-subgroups of $G$ are trivial or noncyclic. See Evans' book \cite{Eva18a}, an expansion of his other book \cite{Eva92a} cited earlier, for a unified proof of the Hall-Paige conjecture. Another proof, using bounds on the number of complete mappings of a given group that are obtained through discrete Fourier analysis, is in the recent paper \cite{EMM22a}.

A \emph{complete mapping}, respectively \emph{orthomorphism}, \emph{of a field $K$} is simply one of the underlying additive group $(K,+)$. Complete mappings and orthomorphisms of \emph{finite} fields have been heavily studied, especially with regard to their polynomial representations, starting with Niederreiter and Robinson's 1982 paper \cite{NR82a}. Later, various practical applications of complete mappings and orthomorphisms of finite fields were discovered, such as in check-digit systems \cite{Sch00a,SW10a} and cryptography \cite{MP14a,SGCGM12a}. Naturally, this spurred even greater interest in them, see \cite{ITW17a,SLGQ21a,TZH14a,Win14a,WLHZ14a,XC15a,ZHC15a}. It should be noted in this context that the most important special case for cryptographic applications is when the finite field in question has characteristic $2$, for which complete mappings are the same as orthomorphisms.

Let us talk some more about the cryptographic applications of orthomorphisms, as this serves to motivate the main results of our paper. There are several reasons why these kinds of functions are valued in cryptography. Many cryptographic protocols involve applying an operation of the form $v\mapsto v+f(v)$, where $v\in\IF_2^m$ may be the cleartext to be encrypted, or a segment of it, or an intermediate result of the encryption process. Moreover, $f$ is some function $\IF_2^m\rightarrow\IF_2^m$, and, naturally, one wants the operation $v\mapsto v+f(v)$ to be injective. While $f$ need not be injective itself (i.e., an orthomorphism) for this to happen, choosing $f$ as an orthomorphism leads to several other desirable properties, such as $f$ being perfectly balanced (Mittenthal, \cite{Mit95a}) and \enquote{usually} having different input and output differentials (Wu and Ye, \cite{WY06a}).

A given cryptographic protocol consists of applying a composition of several so-called \emph{round functions}, parametrized by keys, to the cleartext in order to obtain the ciphertext. In view of what was said in the previous paragraph, it is not surprising that orthomorphisms are ubiquitous in the design of these round functions, and various cryptographic structures, such as the Lai-Massey scheme \cite{Vau99a}, the block cipher FOX \cite{JV05a} and the stream cipher Loiss \cite{FFZFW11a}, explicitly need orthomorphisms to be used in parts of the definition of their round functions in order to ensure certain desirable properties of the structure as a whole.

In addition to this, there are several notable protocols where all round functions themselves are orthomorphisms. For example, in a so-called Feistel cipher (a major class of block ciphers), a given round function $R_{\kappa}:\IF_2^n\rightarrow\IF_2^n$, depending on the key $\kappa$, is defined in terms of a certain other vectorial Boolean function (sometimes called the associated \emph{Feistel transformation}) $F_{\kappa}:\IF_2^{n/2}\rightarrow\IF_2^{n/2}$ as follows: Let $v\in\IF_2^n$, and write $v=(v_{\ell},v_r)$ where $v_{\ell},v_r\in\IF_2^{n/2}$ are the left and right half segments respectively of the string $v$. Then $R_{\kappa}(v)=(v_r,v_{\ell}+F_{\kappa}(v_r))$. This function $R_{\kappa}$ is always a permutation of $\IF_2^n$, with inverse function $(w_{\ell},w_r)\mapsto(w_r+F_{\kappa}(w_{\ell}),w_{\ell})$. Moreover, as observed by Mileva and Markovski in \cite{MM14a}, $R_{\kappa}$ is a complete mapping of $\IF_2^n$ if and only if $F_{\kappa}$ is bijective. Indeed, if $F_{\kappa}$ is bijective, then $\widetilde{R_{\kappa}}:(v_{\ell},v_r)\mapsto (v_{\ell}+v_r,v_{\ell}+v_r+F_{\kappa}(v_r))$ is a permutation of $\IF_2^n$, with inverse function $(u_{\ell},u_r)\mapsto(u_{\ell}+F_{\kappa}^{-1}(u_{\ell}+u_r),F_{\kappa}^{-1}(u_{\ell}+u_r))$. And, on the other hand, if $\widetilde{R_{\kappa}}$ is bijective, then the function $\IF_2^{n/2}\rightarrow\IF_2^n$, $v\mapsto \widetilde{R_{\kappa}}(v,v)$, is injective, which is only possible if $F_{\kappa}$ is injective (and thus bijective). In \cite[Sections 3--5]{MM14a}, three concrete examples of Feistel ciphers all of whose round functions are orthomorphisms are discussed, namely GOST, MIBS and Skipjack. More such ciphers, studied earlier by other authors, are mentioned in \cite[Subsection 1.1]{MM14a}.

Studying the round functions of a cryptographic cipher is key to detecting cryptoanalytic weaknesses of that cipher. One notable line of research, started in 1975 by Coppersmith and Grossman \cite{CG75a}, is concerned with detecting weaknesses of a cipher through properties of the permutation group generated by all round functions of that cipher. In this vein, Kaliski, Rivest and Sherman \cite{KRS88a} observed that this permutation group must not be too small, and Paterson \cite{Pat99a} observed that it should not be imprimitive (i.e., that the round functions should not preserve a common, nontrivial block partition of the underlying group $\IF_2^n$). Paterson's paper motivated various authors to verify for many ciphers that their round functions generate a primitive group (and, if possible, determine this group precisely), see \cite{ACTT18a,ACDS14a,ACS17a,CDS09a,CDS09b,EG83a,HSW94a,MPW94a,SW08a,SW15a,Wer93a}. It should be noted that these group-theoretic conditions are necessary, but not sufficient, for the cipher to be strong -- an example of a weak cipher whose round functions generate the full symmetric group was given by Murphy, Paterson and Wild \cite{MPW94a}.

Motivated by these earlier results, one may ask whether the property that all round functions of a cryptographic cipher are orthomorphisms (or, equivalently if the cipher is over $\IF_2^n$, complete mappings) in and of itself poses a security risk, at least with regard to the group-theoretic conditions of Kaliski, Rivest, Sherman and Paterson. This leads to the following vague question, which we pose for general finite groups:

\begin{quesstion}\label{mainQues1}
Let $G$ be a finite group, and denote by $P_{\comp}(G),P_{\orth}(G)\leq\Sym(G)$ the permutation groups generated by the complete mappings, respectively the orthomorphisms, of $G$. When are these permutation groups \enquote{large} and primitive?
\end{quesstion}

Praeger and Saxl \cite{PS80a} proved that among all primitive permutation groups on a set $X$, the two groups $\Sym(X)$ and $\Alt(X)$ are, asymptotically speaking, significantly larger than all others (more precisely, each primitive group over $X$ that is not isomorphic to one of those two has size less than $4^{|X|}$). So, in fact, it would be desirable to have $\Alt(G)\leq P_{\comp}(G)\cap P_{\orth}(G)$ -- and indeed, one of our main results states that this holds if $G$ is large enough (see Theorem \ref{mainTheo2} below).

For cryptographic applications, the finite elementary abelian groups $\IF_p^n$ (i.e., the underlying additive groups $(K,+)$ of finite fields $K$) are of particular interest, and we were able to determine $P_{\comp}(\IF_p^n)=P_{\comp}(\IF_{p^n})$ and $P_{\orth}(\IF_p^n)$ precisely, see Theorem \ref{mainTheo1} below. At first glance, one may think that earlier results already imply at least that $\Alt(\IF_2^n)\leq P_{\orth}(\IF_2^n)$ for all but a few small $n$ -- for example, as was mentioned above, the round functions of GOST all are orthomorphisms, and in \cite{ACS17a}, it was shown that the round functions of certain so-called \emph{GOST-like ciphers} generate the alternating group. A more detailed look reveals that this is not as easy as it may seem, though, and the authors do not believe that it is obvious from known results that $\Alt(\IF_2^n)\leq P_{\orth}(\IF_2^n)$ even just for infinitely many $n$ -- see Subsection \ref{subsec5P1} for a more detailed analysis.

\begin{theoremm}\label{mainTheo1}
Let $q$ be a prime power. If $q\not=2,3,4,5,8$, then $P_{\comp}(\IF_q)=P_{\orth}(\IF_q)=\Sym(\IF_q)$. For the other values of $q$, the groups $P_{\comp}(\IF_q)$ and $P_{\orth}(\IF_q)$ are as in Table \ref{cTable}. In particular, both of these groups are primitive if $q\geq3$.
\end{theoremm}

\begin{table}[h]
\begin{center}
\begin{tabular}{|c|c|c|}\hline
$q$ & $P_{\comp}(\IF_q)$ & $P_{\orth}(\IF_q)$ \\ \hline
2 & $\{\id\}$ & $\{\id\}$ \\ \hline
3 & $\Alt(\IF_3)$ & $\Sym(\IF_3)$ \\ \hline
4 & $\Alt(\IF_4)$ & $\Alt(\IF_4)$ \\ \hline
5 & $\AGL_1(5)$ & $\AGL_1(5)$ \\ \hline
8 & $\AGL_3(2)$ & $\AGL_3(2)$ \\ \hline
\end{tabular}
\caption{Compositions of complete mappings and of orthomorphisms of finite fields}
\label{cTable}
\end{center}
\end{table}

The proof of Theorem \ref{mainTheo1} is comparatively easy if one uses a recent result of the authors, \cite[Proposition 3.1]{BW22a}, and readers who are solely interested in Theorem \ref{mainTheo1} only need to read up to its proof, which is found after that of Lemma \ref{mainLem2} in Section \ref{sec2}. In the remainder of this Introduction, we discuss our results for general finite groups.

Beside \cite[Theorem]{PS80a}, our aforementioned Theorem \ref{mainTheo2} is also based on a recent asymptotic formula for the number of complete mappings of a finite group $G$ satisfying the Hall-Paige condition obtained by Eberhard, Manners and Mrazovi{\'c} \cite[Theorem 1.2]{EMM22a}. More precisely, those three authors have shown that as $|G|\to\infty$, the number of complete mappings of $G$ is of the form
\begin{equation}\label{complAsymptoticsEq}
(\e^{-1/2}+o(1))\cdot|G/G'|\cdot\frac{(|G|!)^2}{|G|^{|G|}},
\end{equation}
where $G'=\langle g^{-1}h^{-1}gh: g,h\in G\rangle$ denotes the commutator subgroup of $G$. Moreover, they noted that their method could even be used to give an explicit lower bound on the number of complete mappings of $G$ if the corresponding details are worked out.

The precise formulation of our Theorem \ref{mainTheo2} is as follows:

\begin{theoremm}\label{mainTheo2}
Let $c$ be a constant with $0<c<\e^{-1/2}$, and let $G$ be a finite group with at least $c|G/G'|\frac{(|G|!)^2}{|G|^{|G|}}$ complete mappings. If $|G|\geq\max\{280816,11\log\left(c^{-1}\right)\}$, then $\Alt(G)\leq P_{\comp}(G)\cap P_{\orth}(G)$.
\end{theoremm}

The following is clear from Theorem \ref{mainTheo2} and \cite[Theorem 1.2]{EMM22a}:

\begin{corrollary}\label{mainCor}
Let $G$ be a sufficiently large finite group that satisfies the Hall-Paige condition. Then $\Alt(G)\leq P_{\comp}(G)\cap P_{\orth}(G)$.
\end{corrollary}

Since $\Alt(G)$ consists precisely of the even permutations in $\Sym(G)$, Corollary \ref{mainCor} implies that there is an absolute constant $N_0$ (which could be worked out from Eberhard, Manner and Mrazovi{\'c}'s methods) such that for all finite groups $G$ with $|G|\geq N_0$ that satisfy the Hall-Paige condition,
\[
P_{\comp}(G)=\begin{cases}\Sym(G), & \text{if }G\text{ has an odd complete mapping}, \\ \Alt(G), & \text{otherwise},\end{cases}
\]
and
\[
P_{\orth}(G)=\begin{cases}\Sym(G), & \text{if }G\text{ has an odd orthomorphism}, \\ \Alt(G), & \text{otherwise}.\end{cases}
\]
This motivates a closer study of the parities of complete mappings and orthomorphisms. A remarkable recent result in this context is M{\"u}yesser and Pokrovskiy's \cite[Theorem 6.9]{MP22a}, which states that every large enough finite group $G$ that satisfies the Hall-Paige condition and is not an elementary abelian $2$-group is \emph{harmonious} in the sense that it has a \emph{harmonious ordering} (as defined in \cite[Subsubsection 1.1.3]{MP22a}): a repetition-free ordered list $(g_1,g_2,\ldots,g_{|G|})$ of the elements of $G$ such that the $|G|$ group elements $g_ig_{i+1}$ for $1\leq i\leq |G|-1$ and $g_ng_1$ are pairwise distinct. But this is equivalent to the permutation of $G$ that consists of the single $|G|$-cycle $(g_1,g_2,\ldots,g_{|G|})$ being a complete mapping of $G$. In particular, if, additionally, $G$ is of even order, then $G$ has an odd complete mapping.

Another useful observation is that the following condition on a finite group $G$, which we will henceforth refer to as \emph{property (P)} (the \enquote{P} stands for \enquote{parity}), is well-suited for inductive purposes:
\begin{equation}\label{condPEq}
G\text{ admits both even and odd complete mappings.}
\end{equation}
We note that $G$ satisfies property (P) if and only if $G$ admits orthomorphisms of both parities: Indeed, if $\inv:g\mapsto g^{-1}$ denotes the inversion function of $G$, and if $f_1$ and $f_2$ are complete mappings of $G$ of different parities, then $f_1\circ\inv$ and $f_2\circ\inv$ are orthomorphisms of $G$ of different parities.

The following lemma, which can be obtained by studying Hall and Paige's methods of constructing complete mappings from \cite[Theorem 1]{HP55a}, is a powerful tool for inductive proofs involving property (P):

\begin{lemmma}\label{mainLem}
Let $G$ be a finite group and $N$ a normal subgroup of $G$. Assume that the following hold:
\begin{enumerate}
\item $N$ satisfies property (P).
\item $G/N$ admits a complete mapping.
\end{enumerate}
Then $G$ satisfies property (P).
\end{lemmma}

Apart from being used in the proof of Theorem \ref{mainTheo1}, Lemma \ref{mainLem} (and a generalization of it which we will formulate as Lemma \ref{mainLemVar}) can be used to infer the following:

\begin{theoremm}\label{mainTheo4}
Let $G$ be a finite group, and assume that at least one of the following holds:
\begin{enumerate}
\item $G$ is of odd order and $|G|>3$.
\item $G$ is a noncyclic $2$-group and $|G|>8$.
\item $G$ is solvable, satisfies the Hall-Paige condition, and $|G|>24$.
\end{enumerate}
Then $G$ satisfies property (P). In particular:
\begin{enumerate}\setcounter{enumi}{3}
\item If $G$ is a large enough finite group that satisfies the Hall-Paige condition, then $P_{\comp}(G)=\Sym(G)$.
\item If $G$ is a large enough finite \emph{solvable} group that satisfies the Hall-Paige condition, then $P_{\orth}(G)=\Sym(G)$.
\end{enumerate}
\end{theoremm}

The proof of Theorem \ref{mainTheo4} requires us to verify property (P) for certain special cases, which leads to results of independent interest, such as that groups of the form
\[
\langle x,y: x^{2^{n-1}}=y^2=1, y^{-1}xy=x^{2^{n-2}+1}\rangle
\]
are harmonious. As an application of Theorem \ref{mainTheo4}, we will determine the possible values of two kinds of parity types (in the sense of \cite[Introduction]{FHW18a} and \cite[Definition 1]{Kot12a} respectively) of orthogonal Latin squares based on complete mappings of finite groups. Some open questions raised by Theorem \ref{mainTheo4} are discussed at the end of the paper.

\subsection{Overview of the paper}\label{subsec1P2}

Section \ref{sec2} is basically dedicated to Lemma \ref{mainLem} and applications of it. We will, however, actually prove a more general version of Lemma \ref{mainLem}, formulated as Lemma \ref{mainLemVar}; this stronger version of the lemma will be needed in Section \ref{sec4}. The remainder of Section \ref{sec2} consists of proofs of Theorems \ref{mainTheo1} and \ref{mainTheo4}(1), both of which make use of Lemma \ref{mainLem}.

In Section \ref{sec3}, we prove Theorem \ref{mainTheo2} and infer Theorem \ref{mainTheo4}(4) from it and the results of Section \ref{sec2}.

Section \ref{sec4} deals with the proof of Theorem \ref{mainTheo4}(2,3,5). First, we show that statements (3) and (5) of Theorem \ref{mainTheo4} are simple consequences of statements (1), which was already proved in Section \ref{sec2}, and (2). This observation uses Lemma \ref{mainLemVar}. Therefore, in the remainder of Section \ref{sec4}, we will be concerned with proving Theorem \ref{mainTheo4}(2), which requires an analysis of the parities of complete mappings of finite $2$-groups. As this, too, uses Lemma \ref{mainLemVar}, it is not surprising that the proof proceeds by induction on the group order, and the induction base is covered by two key auxiliary results, Propositions \ref{specialCasesProp} and \ref{orders16and32Prop}. These deal with (arbitarily large) $2$-groups of a particular form (that cannot be reduced to groups of smaller order by the general inductive approach), respectively with groups of orders $16$ and $32$. The proofs of these auxiliary results are technical at times, and the more technically demanding parts of the arguments are deferred to the four Appendices at the end of the paper for the sake of the reading flow.

Finally, Section \ref{sec5} provides some concluding remarks. More precisely, in Subsection \ref{subsec5P1} we take a closer look at the relationship of Theorem \ref{mainTheo1} with known results in cryptography. Subsection \ref{subsec5P2} discusses a connection between parities of complete mappings (and orthomorphisms) and those of orthogonal Latin squares as studied in combinatorics (see \cite{FHW18a} and \cite{Kot12a}, for instance). Finally, Subsection \ref{subsec5P3} concludes the paper with a discussion of some related open problems for further research.

\section{Property (P) and applications}\label{sec2}

In this section, we will prove Lemma \ref{mainLem} as well as Theorems \ref{mainTheo1} and \ref{mainTheo4}(1). As for Lemma \ref{mainLem}, we will actually prove the following technical result (needed in our discussion of solvable groups in Section \ref{sec4}), of which Lemma \ref{mainLem} is a simple consequence:

\begin{lemma}\label{mainLemVar}
Let $G$ be a group, $H$ a subgroup of $G$ of index $k$, and let $u_1,u_2,\ldots,u_k$ be elements of $G$ that form both a left and right transversal of $H$ in $G$, with $u_1\in H$. Moreover, assume that there exist permutations $S,T\in\Sym(k)=\Sym(\{1,2,\ldots,k\})$ such that $S(1)=T(1)=1$ and for each $i\in\{1,2,\ldots,k\}$, we have $u_iu_{S(i)}H=u_{T(i)}H$. If $H$ admits a complete mapping, then $G$ admits a complete mapping that stabilizes $H$. In particular, if, additionally, $H$ satisfies property (P), then so does $G$.
\end{lemma}

\begin{proof}
The main statement follows from Hall and Paige's construction in the proof of \cite[Theorem 1]{HP55a}. Indeed, using their notation and our extra assumptions compared to their theorem, note that $u_{[1,p]}=u_1$ for all $p\in H$, whence the complete mapping $\Theta$ which they define in \cite[formula (3) on p.~542]{HP55a} stabilizes $H$.

For the \enquote{In particular}, let $f$ be a complete mapping of $G$ that stabilizes $H$. The restriction $f_{\mid H}$ is a complete mapping of $H$. Let $h$ be a complete mapping of $H$ of different parity than $f_{\mid H}$. Then the function
\[
f':G\rightarrow G, f'(x)=\begin{cases}h(x), & \text{if }x\in H, \\ f(x), & \text{if }x\notin H,\end{cases}
\]
is a complete mapping of $G$ of different parity than $f$.
\end{proof}

\begin{proof}[Proof of Lemma \ref{mainLem}]
Let $k:=|G:N|$, and let $u_1,u_2,\ldots,u_k$ be a transversal of $N$ in $G$ with $u_1=1_G\in N$. The $u_i$ are in bijection with the elements of $G/N$ via the canonical projection $\pi:G\rightarrow G/N$. Let $g$ be a complete mapping of $G/N$. Note that if $\rho:G/N\rightarrow\Sym(G/N)$ denotes the right-regular representation of $G/N$ on itself (so that $\rho(x)(x')=x'x$ for all $x,x'\in G/N$), then the function $\rho(x)\circ g$, mapping $a\in G/N$ to $g(a)x$, is also a complete mapping of $G/N$. We may thus assume without loss of generality that $g(1_{G/N})=1_{G/N}$ -- if not, just replace $g$ by $\rho(g(1_{G/N})^{-1})\circ g$. If we define $S,T:\{1,2,\ldots,k\}\rightarrow\{1,2,\ldots,k\}$ via $u_{S(i)}N=g(u_iN)$ and $u_{T(i)}N=u_iu_{S(i)}N$, then $S$ and $T$ are permutations in $\Sym(k)$ as in the conditions of Lemma \ref{mainLemVar}, whence the result follows by an application of that lemma.
\end{proof}

In order to derive Theorem \ref{mainTheo1}(1) from Lemma \ref{mainLem}, we first derive the following auxiliary result:

\begin{lemma}\label{mainLem2}
Let $q=p^d$ be a prime power. If $q>4$, then $\AGL_d(p)$ is a subgroup of $P_{\comp}(\IF_p^d)\cap P_{\orth}(\IF_p^d)$.
\end{lemma}

\begin{proof}
By \cite[Proposition 3.1]{BW22a}, each element of $\GL_d(p)$ is a product of two complete mappings in $\GL_d(p)$, so certainly $\GL_d(p)\leq P_{\comp}(\IF_p^d)$. Moreover, if $g\in\GL_d(p)$ is written as $g=h_1h_2$ where $h_1,h_2\in\GL_d(p)$ are complete mappings of $\IF_p^d$, then also $g=(-h_1)(-h_2)$, and $-h_1,-h_2\in\GL_d(p)$ are orthomorphisms. Hence $\GL_d(p)\leq P_{\orth}(\IF_p^d)$ as well. But if $\rho$ is the right-regular representation of $\IF_p^d$ on itself, then for each $x\in\IF_p^d$ and each complete mapping, respectively orthomorphism, $f$ of $\IF_p^d$, the composition $\rho(x)\circ f$ is also a complete mapping, respectively orthomorphism, of $\IF_p^d$. Hence $\rho(\IF_p^d)\leq P_{\comp}(\IF_p^d)\cap P_{\orth}(\IF_p^d)$ as well, and since $\AGL_d(p)=\langle\GL_d(p),\rho(\IF_p^d)\rangle$, the result follows.
\end{proof}

\begin{proof}[Proof of Theorem \ref{mainTheo1}]
The field $\IF_2$ has no complete mappings, and orthomorphisms of it are the same as complete mappings, whence
\[
P_{\comp}(\IF_2)=P_{\orth}(\IF_2)=\langle\emptyset\rangle_{\Sym(\IF_2)}=\{\id\}.
\]
For $\IF_3$, it follows by comparing group orders that
\[
\Sym(\IF_3)=\AGL_1(3)=\{x\mapsto ax+b: a\in\IF_3^{\ast},b\in\IF_3\}.
\]
The function $x\mapsto ax+b$ is a complete mapping, respectively an orthomorphism, of $\IF_3$ if and only if $a=1$, respectively $a=-1$. Therefore, the set of complete mappings of $\IF_3$ is equal to
\[
\{x\mapsto x+b: b\in\IF_3\}=\rho(\IF_3)=\Alt(\IF_3),
\]
an index $2$ subgroup of $\Sym(\IF_3)$, and the set of orthomorphisms of $\IF_3$ is the complement of $\Alt(\IF_3)$ in $\Sym(\IF_3)$. It follows that $P_{\comp}(\IF_3)=\Alt(\IF_3)$ and $P_{\orth}(\IF_3)=\Sym(\IF_3)$, as asserted in Table \ref{cTable}.

As for $\IF_4$ (for which complete mappings and orthomorphisms are the same), observe that each complete mapping $f$ of a finite field of characteristic $2$ has
\begin{itemize}
\item exactly one fixed point, since a fixed point of $f$ is the same as a pre-image of $0$ under $\tilde{f}$, and
\item no $2$-cycles, because if $x$ lies on a $2$-cycle of $f$, then $\tilde{f}(x)=x+f(x)=f(f(x))+f(x)=f(x)+f(f(x))=\tilde{f}(f(x))$, contradicting the injectivity of $\tilde{f}$.
\end{itemize}
Therefore, all complete mappings of $\IF_4\cong\IF_2^2=\langle e_1,e_2\rangle$ are $3$-cycles. Hence, a complete mapping of $\IF_4$ that fixes $0$ must be equal to one of $(e_1,e_2,e_1+e_2)$ or $(e_1,e_1+e_2,e_2)=(e_1,e_2,e_1+e_2)^{-1}$. It is not hard to check that $(e_1,e_2,e_1+e_2)$ is a complete mapping of $\IF_4$, whence
\[
P_{\comp}(\IF_4)=P_{\orth}(\IF_4)=\langle(e_1,e_2,e_1+e_2),\rho(\IF_4)\rangle=\Alt(\IF_4).
\]

We may thus assume that $q>4$. Writing $q=p^d$, we have $\AGL_d(p)\leq P_{\comp}(\IF_q)\cap P_{\orth}(\IF_q)$ by Lemma \ref{mainLem2}, and so each of $P_{\comp}(\IF_q)$ and $P_{\orth}(\IF_q)$ must be one of the following according to \cite[Theorem 2]{Sta98a}:
\begin{itemize}
\item $\AGL_d(p)$ or $\Sym(\IF_q)$ if $p>2$;
\item $\AGL_d(p)$, $\Alt(\IF_q)$ or $\Sym(\IF_q)$ if $p=2$.
\end{itemize}
Let us consider the cases \enquote{$p>2$} and \enquote{$p=2$} separately. If $p>2$ and $q\geq 13$ or $q=7$, then \cite[Corollary 1 on p.~206]{NR82a} shows that there is a non-linearized complete permutation polynomial over $\IF_q$. In other words, there is a complete mapping $f$ of $\IF_q$ that does not lie in $\AGL_d(p)$, and so $P_{\comp}(\IF_q)=\Sym(\IF_q)$ necessarily. Moreover, $\tilde{f}=f+\id$ is an orthomorphism of $\IF_q$ that does not lie in $\AGL_d(p)$, whence $P_{\orth}(\IF_q)=\Sym(\IF_q)$ as well. As for $q=5$, one can enumerate the complete mappings of $\IF_5$ completely with a computer (we used GAP \cite{GAP4} for this) to verify that all complete mappings, and thus also all orthomorphisms, of $\IF_5$ lie in $\AGL_1(5)$, which shows that $P_{\comp}(\IF_5)=P_{\orth}(\IF_5)=\AGL_1(5)$. Finally, for $q=9,11$, it is not hard to find a non-additive complete mapping $f$ of $\IF_q$ with $f(0)=0$ through random search (for which we also used GAP):
\begin{itemize}
\item For $\IF_9\cong\IF_3^2=\langle e_1,e_2\rangle$:
\[
f=(0)(2e_2)(e_1,e_1+2e_2,e_2,2e_1+2e_2,2e_1,2e_1+e_2,e_1+e_2).
\]
Note that $f$ fixes $2e_2$ but not $e_2=2(2e_2)$, so $f$ cannot be additive.
\item For $\IF_{11}=\{\overline{k}:k=0,1,\ldots,10\}$:
\[
f=(\overline{0})(\overline{6})(\overline{7})(\overline{9})(\overline{10})(\overline{1},\overline{4},\overline{2},\overline{8},\overline{5},\overline{3}).
\]
Note that $f$ fixes $\overline{6}$ but not $\overline{1}=2\cdot\overline{6}$, so $f$ cannot be additive.
\end{itemize}
Hence $P_{\comp}(\IF_q)=\Sym(\IF_q)$ for those two values of $q$. Since $\tilde{f}=f+\id$ is a non-additive orthomorphism in both cases, we also have $P_{\orth}(\IF_q)=\Sym(\IF_q)$.

Let us now assume that $p=2$ (and $q>4$), and recall that complete mappings are the same as orthomorphisms for those fields. By enumerating the complete mappings of $\IF_8$ using GAP, one can verify that they all are $\IF_2$-affine, whence $P_{\comp}(\IF_8)=P_{\orth}(\IF_8)=\AGL_3(2)$. On the other hand, $\IF_{16}\cong\IF_2^4=\langle e_1,e_2,e_3,e_4\rangle$ has the following odd complete mapping, found through a random search with GAP:
\begin{align*}
&(e_1+e_3,e_2+e_3+e_4,e_1+e_2+e_3)(e_4) \\
&(0,e_3+e_4,e_2+e_3,e_1,e_1+e_4,e_3,e_2+e_4,e_1+e_2+e_3+e_4,e_1+e_2+e_4,e_1+e_3+e_4, \\
&e_2,e_1+e_2).
\end{align*}
Since neither $\AGL_4(2)$ nor $\Alt(\IF_{16})$ contain any odd permutations, this shows that $P_{\comp}(\IF_{16})=P_{\orth}(\IF_{16})=\Sym(\IF_{16})$. Similarly, odd complete mappings of $\IF_{32}$ were found by Schimanski in \cite[table on pages 82--91]{Sch16a} as part of an extensive determination of possible cycle types of complete mappings of that field -- see for example the second entry of the cited table (on page 82 of \cite{Sch16a}), and note that Schimanski encoded each element $(\epsilon_1,\ldots,\epsilon_5)\in\IF_2^5$ through the integer number in $\{0,1,\ldots,31\}$ with binary digit representation $(\epsilon_1\epsilon_2\cdots\epsilon_5)_2$. We infer that $P_{\comp}(\IF_{32})=P_{\orth}(\IF_{32})=\Sym(\IF_{32})$.

Finally, let $q=2^d\geq64$ be an even prime power, and let $N$ be a normal subgroup of the additive group of $\IF_q$ with $|N|=16$. Then $N$ satisfies property (P) -- we specified an odd complete mapping above, and any Singer cycle is an even complete mapping -- and $\IF_q/N\cong\IF_2^{d-4}\geq\IF_2^2$ has a complete mapping. Therefore, $\IF_q$ satisfies property (P) by Lemma \ref{mainLem}. In particular, $\IF_q$ has an odd complete mapping, and so $P_{\comp}(\IF_q)=P_{\orth}(\IF_q)=\Sym(\IF_q)$.
\end{proof}

\begin{proof}[Proof of Theorem \ref{mainTheo4}(1)]
If $G$ is of prime order (larger than $3$), then $G$ satisfies property (P), because $G$ must have odd complete mappings by Theorem \ref{mainTheo1}, while $\id_G$ is an even complete mapping of $G$. We may thus assume that $|G|$ is \emph{not} a prime, and we proceed by induction on $|G|$.

First, consider the case where $G$ is abelian. If $G$ is not a $3$-group, then $|G|$ is divisible by some prime $p>3$, and so $G$ has a (normal) subgroup $N\cong\IZ/p\IZ$. This normal subgroup $N$ satisfies property (P), and $G/N$ has a complete mapping, so $G$ satisfies property (P) by Lemma \ref{mainLem}.

We may thus assume that the abelian group $G$ \emph{is} a $3$-group. Then $|G|=3^d$ for some $d\geq2$. If $d=2$, then $G\cong(\IZ/3\IZ)^2$ or $G\cong\IZ/9\IZ$. In the former case, $G$ has complete mappings of both parities (the identity function and any Singer cycle of $G$ are an even and odd complete mapping respectively), and in the latter case, note that the identity function of $\IZ/9\IZ=\{\overline{k}:k=0,1,\ldots,8\}$ is an even complete mapping, whereas the following permutation (which we found through a random search with GAP \cite{GAP4}) is an odd complete mapping:
\[
(\overline{0},\overline{1},\overline{3},\overline{6},\overline{8})(\overline{2},\overline{4},\overline{7},\overline{5}).
\]
We may thus assume that $d\geq3$. Then $G$ has a (normal) subgroup $N$ with $|N|=9$. Since $N$ satisfies property (P), and $G/N$ has a complete mapping, we conclude that $G$ satisfies property (P) by Lemma \ref{mainLem}.

Now assume that $G$ is nonabelian. Then the commutator subgroup $G'$ is nontrivial (and proper since $G$ is solvable). If $|G'|>3$, we are done by our induction hypothesis (applied to $G'$) and Lemma \ref{mainLem}. Hence, assume that $|G'|=3$. If the (abelian) group $G/G'$ is not of prime order, then $G/G'$ has a proper nontrivial normal subgroup $\overline{N}$, which can be lifted to a proper nontrival normal subgroup $N$ of $G$ with $|N|>|G'|=3$, and we are again done by the induction hypothesis (applied to $N$) and Lemma \ref{mainLem}. It remains to consider the case where $|G/G'|=p$ for some (odd) prime $p$. Note that $p>3$, since otherwise, $|G|=9$ and $G$ is abelian, contradicting our case assumption. Since $|G|=3p$ and $G$ has a cyclic subgroup $H$ of order $p$, we conclude that $G=H\ltimes G'$. Since $H$ is of odd order and $\Aut(G')\cong\Aut(\IZ/3\IZ)\cong\IZ/2\IZ$, no element of $H$ can act nontrivially on $G'$ in this semidirect product. Hence the semidirect product in question is actually a direct product: $G=H\times G'$. But this implies that $G$ is abelian, again contradicting our case assumption.
\end{proof}

\section{The groups generated by all complete mappings or orthomorphisms}\label{sec3}

In this section, we will prove Theorems \ref{mainTheo2} as well as \ref{mainTheo4}(4), starting with the former:

\begin{proof}[Proof of Theorem \ref{mainTheo2}]
Let $G$ be a finite group of order $n$ satisfying the assumptions of Theorem \ref{mainTheo2}. Our goal is to show that both $P_{\comp}(G)$ and $P_{\orth}(G)$ contain $\Alt(G)$. The proof is analogous in both cases; let $P(G)\in\{P_{\comp}(G),P_{\orth}(G)\}$, and assume that $\Alt(G)$ is \emph{not} contained in $P(G)$. Since $G$ has as many complete mappings as it has orthomorphisms (with $f\mapsto f\circ\inv$ being a bijection in either direction), it follows that
\[
|P(G)|\geq c|G/G'|\frac{(|G|!)^2}{|G|^{|G|}}.
\]
We distinguish the following three cases:
\begin{enumerate}
\item $P(G)$ is intransitive.
\item $P(G)$ is transitive, but imprimitive.
\item $P(G)$ is primitive.
\end{enumerate}

Assume that case (1) applies. If $f$ is a complete mapping, respectively an orthomorphism, of $G$, then for each $g\in G$, the function $G\rightarrow G$, $x\mapsto f(x)g$, is also a complete mapping, respectively an orthomorphism, of $G$. But this function can be written as the composition $\rho(g)\circ f$ where $\rho(g):G\rightarrow G, x\mapsto xg$, is the image of $g$ under the right-regular representation $\rho$ of $G$ on itself. It follows that $\rho(G)\leq P(G)$, contradicting the intransitivity of $P(G)$.

Now assume that case (3) applies. Since $P(G)\not=\Alt(G),\Sym(G)$, we may use \cite[Theorem]{PS80a} to conclude that $|P(G)|<4^n$. Hence
\[
c\frac{n^n}{\e^{2n}}<c\frac{(n!)^2}{n^n}\leq c|G/G'|\frac{(n!)^2}{n^n}<4^n,
\]
where the first inequality holds because of the following Stirling-like bound due to Robbins \cite{Rob55a}:
\[
n!>\e^{\frac{1}{12n+1}}\sqrt{2\pi n}\left(\frac{n}{\e}\right)^n\geq\left(\frac{n}{\e}\right)^n.
\]
Taking logarithms on both sides of the inequality
\[
c\frac{n^n}{\e^{2n}}<4^n,
\]
we obtain
\[
\log{c}+n\log{n}-2n<n\log{4},
\]
or, equivalently,
\[
n(\log{4}+2-\log{n})>\log{c}.
\]
However, using that $n\geq\max\{280816,-11\log{c}\}$, we obtain
\[
n(\log{4}+2-\log{n})<(-11\log{c})\cdot(-9)=99\log{c}<\log{c},
\]
a contradiction.

Finally, assume that case (2) applies. Let $\Bcal=\{B_1,B_2,\ldots,B_m\}$ be a nontrivial block system of $G$ preserved by $P(G)$. As noted in the argument for case (1), we have $\rho(G)\leq P(G)$, so $\rho(G)$ preserves $\Bcal$ as well. But the only block systems of $G$ preserved by $\rho(G)$ are the decompositions of $G$ into the right cosets of a given subgroup. Indeed, the block containing $1_G$, say $B_1$, is a subgroup of $G$: If $g\in B_1$, then $\rho(g)$ maps $1_G\in B_1$ to $g\in B_1$, whence $\rho(g)(B_1)=B_1$. In particular, if $g,g'\in B_1$, then $gg'=\rho(g')(g)\in B_1$ and $g^{-1}=\rho(g)^{-1}(1_G)\in B_1$. Since each $B_i$ is of the form $\rho(g_i)(B_1)=B_1g_i$ for a suitable $g_i\in G$, and since the blocks $B_i$ partition $G$, $\Bcal$ is the set of all right cosets of $B_1$ in $G$, as required.

For simplicity, write $H$ for the block containing $1_G$. Since all $m$ blocks are of the same size, note that $|H|=\frac{|G|}{m}$ is a proper, nontrivial divisor of $|G|$.

First, assume that $m=|G:H|=2$, i.e., that $\Bcal=\{H,G\setminus H\}$. Let $f$ be a complete mapping, respectively orthomorphism, of $G$. Since $f$ preserves $\Bcal$, it induces a complete mapping, respectively orthomorphism, on $G/H=\Bcal$. But $G/H\cong\IZ/2\IZ$ does not have any complete mappings nor orthomorphisms, a contradiction.

So we may henceforth assume that $m\geq 3$. For ease of notation, set $k:=|H|$, so that $n=|G|=m\cdot k$. Note that by assumption, $P(G)$ is contained in a (maximal) subgroup $M$ of $\Sym(G)$ that is isomorphic (as a permutation group) to the imprimitive permutational wreath product $\Sym(k)\wr_{\imp}\Sym(m)$. In particular, $|M|=(k!)^m\cdot m!$. We distinguish two cases:
\begin{itemize}
\item Case: $k\leq n^{\frac{1}{3}}$. Then, using Robbins' \cite{Rob55a} Stirling-like bound
\[
\ell!\leq\e^{\frac{1}{12}}\sqrt{2\pi\ell}\left(\frac{\ell}{\e}\right)^{\ell}\text{ for all }\ell\geq1,
\]
as well as that $k\geq2$ and $n\geq6$, we conclude that
\begin{align*}
c\frac{n^n}{\e^{2n}}&\leq|P(G)|\leq|M|=(k!)^m\cdot m!\leq\e^{\frac{1}{12}(m+1)}(2\pi k)^{m/2}(2\pi m)^{1/2}\left(\frac{k}{\e}\right)^n\left(\frac{m}{\e}\right)^m \\
&=\e^{\frac{1}{12}\left(\frac{n}{k}+1\right)}(2\pi k)^{\frac{n}{2k}}\left(2\pi\frac{n}{k}\right)^{\frac{1}{2}}\left(\frac{k}{\e}\right)^n\left(\frac{n}{\e k}\right)^{\frac{n}{k}} \\
&=\e^{\frac{1}{12}\left(\frac{n}{k}+1\right)}\cdot\left(2\pi\right)^{\frac{n}{2k}+\frac{1}{2}}\cdot\e^{-\left(1+\frac{1}{k}\right)n}\cdot k^{n+\frac{n}{2k}-\frac{1}{2}-\frac{n}{k}}\cdot n^{\frac{1}{2}+\frac{n}{k}} \\
&\leq\e^{\frac{1}{12}}(2\pi)^{\frac{1}{2}}\left(\frac{\e^{1/24}(2\pi)^{1/4}}{\e}\right)^n\cdot n^{\frac{1}{3}\left(n-\frac{1}{2}-\frac{n}{2k}\right)}\cdot n^{\frac{n}{2}+\frac{1}{2}} \\
&\leq1\cdot n^{\left(\frac{1}{3}+\frac{1}{2}\right)n}\cdot n^{-\frac{2}{3}+\frac{1}{2}}\leq n^{\frac{5}{6}n}.
\end{align*}
But $n^{\frac{5}{6}n}<c\frac{n^n}{\e^{2n}}$ is equivalent to $c^{-1}\e^{2n}<n^{\frac{1}{6}{n}}$, and further (through taking logarithms on both sides) to $\log\left(c^{-1}\right)+2n<\frac{1}{6}n\log{n}$, which simplifies to
\[
n>\frac{1}{\frac{1}{6}\log{n}-2}\log\left(c^{-1}\right),
\]
which is true since $(\frac{1}{6}\log{n}-2)^{-1}<11$ due to $n\geq 280816$.
\item Case: $k>n^{\frac{1}{3}}$. Note that we ruled out the case $m=\frac{n}{k}=2$ above. Assume that $m\in\{3,4,5\}$. Since $n\geq459$, we have
\begin{align*}
c\frac{n^n}{\e^{2n}}&\leq|P(G)|\leq|M|=(k!)^m\cdot m!\leq 120\cdot(k!)^m\leq 120\e^{\frac{m}{12}}\left(2\pi k\right)^{\frac{m}{2}}\left(\frac{k}{\e}\right)^n \\
&\leq 120\e^{\frac{5}{12}}\left(2\pi\right)^{\frac{5}{2}}k^{n+\frac{5}{2}}\cdot\frac{1}{\e^n}=120\e^{\frac{5}{12}}\left(2\pi\right)^{\frac{5}{2}}\left(\frac{n}{m}\right)^{n+\frac{5}{2}}\cdot\frac{1}{\e^n} \\
&\leq120\e^{\frac{5}{12}}\left(2\pi\right)^{\frac{5}{2}}\frac{1}{3^{5/2}}n^{\frac{5}{2}}\cdot\left(\frac{n}{3\e}\right)^n\leq\left(\frac{201}{200}\right)^n\cdot\left(\frac{n}{3\e}\right)^n=\left(\e\cdot\frac{67}{200}\right)^n\cdot\frac{n^n}{\e^{2n}},
\end{align*}
which is less than $c\frac{n^n}{\e^{2n}}$ (and thus yields a contradiction) because
\[
n>11\log\left(c^{-1}\right)>-\frac{1}{\log\left(\e\cdot\frac{67}{200}\right)}\log\left(c^{-1}\right).
\]

Now assume that $m\geq6$. Note that because $k>n^{\frac{1}{3}}$, we have $m<n^{\frac{2}{3}}$. Since $|M|=(k!)^m\cdot m!$, we get (by taking logarithms on both sides)
\begin{align*}
\log{|M|}&=m\cdot\log\left(k!\right)+\log\left(m!\right)\leq\frac{n}{k}\cdot\log\left(\e^{\frac{1}{12}}\sqrt{2\pi k}\left(\frac{k}{\e}\right)^k\right)+\log{m^m} \\
&\leq\frac{n}{k}\cdot\left(\frac{1}{12}+\frac{1}{2}\log\left(2\pi\right)+\frac{1}{2}\log{k}+k\log{k}-k\right)+n^{\frac{2}{3}}\cdot\frac{2}{3}\log{n} \\
&<n\cdot\left(\log{k}-\frac{1}{2}\right)+\frac{2}{3}n^{\frac{2}{3}}\log{n}=n\cdot\left(\log{\frac{n}{m}}-\frac{1}{2}\right)+\frac{2}{3}n^{\frac{2}{3}}\log{n} \\
&=n\log{n}-\frac{n}{2}-n\log{m}+\frac{2}{3}n^{\frac{2}{3}}\log{n} \\
&=n\log{n}-n\cdot\left(\frac{1}{2}+\log{m}-\frac{2}{3}n^{-\frac{1}{3}}\log{n}\right) \\
&\leq n\log{n}-n\cdot\left(\frac{1}{2}+\log{6}-\frac{2}{3}n^{-\frac{1}{3}}\log{n}\right)\leq n\log{n}-2.16n,
\end{align*}
where the strict inequality in this chain uses that
\[
\frac{1}{12}+\frac{1}{2}\log\left(2\pi\right)+\frac{1}{2}\log{k}<\frac{k}{2},
\]
which is true because $k\geq4$ due to $n\geq64$. Applying the exponential function to both sides of the derived inequality $\log{|M|}\leq n\log{n}-2.16n$, we conclude that
\[
|M|\leq\frac{n^n}{\e^{2.16n}}=\frac{1}{\e^{0.16n}}\cdot\frac{n^n}{\e^{2n}},
\]
and this is less than $c\frac{n^n}{\e^{2n}}$ (thus yielding a contradiction) since
\[
n>11\log\left(c^{-1}\right)>0.16^{-1}\log\left(c^{-1}\right).
\]
\end{itemize}
\end{proof}

\begin{proof}[Proof of Theorem \ref{mainTheo4}(4)]
Let $G$ be a sufficiently large finite group that satisfies the Hall-Paige condition. Then $\Alt(G)\leq P_{\comp}(G)$ by \cite[Theorem 1.2]{EMM22a} and Theorem \ref{mainTheo2}, which we proved just above. We want to show that $P_{\comp}(G)=\Sym(G)$, which holds if and only if $G$ has an odd complete mapping, and we will verify this to be true in the following case distinction:
\begin{enumerate}
\item Case: $G$ is of odd order. Then $G$ satisfies property (P) by Theorem \ref{mainTheo4}(1), which was proved in the previous section. In particular, $G$ has an odd complete mapping, as required.
\item Case: $G$ is an elementary abelian $2$-group. Then Theorem \ref{mainTheo1}, which was also proved in the previous section, directly states that $P_{\comp}(G)=\Sym(G)$.
\item Case: $G$ is of even order, but not an elementary abelian $2$-group. By \cite[Theorem 6.9]{MP22a}, $G$ has a complete mapping $f$ that permutes the elements of $G$ in a single cycle. In particular, $f$ is an odd complete mapping of $G$, as required.
\end{enumerate}
\end{proof}

\section{Solvable groups}\label{sec4}

In order to deal with statements (2), (3) and (5) of Theorem \ref{mainTheo4}, we follow Hall and Paige's approach from \cite[proofs of Lemma 1 and Theorem 4]{HP55a}. If $G$ is a finite solvable group, and if $S$ is a Sylow $2$-subgroup of $G$ with Hall complement $H$, then $G=S\cdot H=H\cdot S$. This means that each of $S$ and $H$ forms a left and right transversal for the respective other in $G$, and the following can be easily derived from Lemma \ref{mainLemVar}:

\begin{proposition}\label{solvableProp}
If both $S$ and $H$ admit complete mappings, and at least one of them satisfies property (P), then $G$ satisfies property (P).
\end{proposition}

Using Proposition \ref{solvableProp}, we can easily derive statements (3) and (5) in Theorem \ref{mainTheo4} from its statements (1) and (2), the latter of which we still need to prove:

\begin{proof}[Proof of Theorem \ref{mainTheo4}(3,5) assuming Theorem \ref{mainTheo4}(2) to be true]
First, let us discuss the proof of statement (3). If $G=S\cdot H$ is a finite solvable group factored as above and $|G|>24$, then at least one of the two inequalities $|S|>8$ and $|H|>3$ must hold. Since $G$ satisfies the Hall-Paige condition by assumption, it follows that each of $S$ and $H$ admits a complete mapping, and by Theorem \ref{mainTheo4}(1,2), at least one of them satisfies property (P). Hence Proposition \ref{solvableProp} can be applied to conclude that $G$ satisfies property (P), as required.

Now we prove statement (5). By \cite[Theorem 1.2]{EMM22a} and Theorem \ref{mainTheo2}, we know that $\Alt(G)\leq P_{\orth}(G)$, so we will have $P_{\orth}(G)=\Sym(G)$ as long as $G$ has an odd orthomorphism. But by statement (3) of Theorem \ref{mainTheo4}, proved just above, $G$ satisfies property (P), so it does have an odd orthomorphism.
\end{proof}

We can thus focus our attention on Theorem \ref{mainTheo4}(2), i.e., on $2$-groups. Propositions \ref{specialCasesProp} and \ref{orders16and32Prop} below are key auxiliary results.

\begin{proposition}\label{specialCasesProp}
Each of the following finite $2$-groups satisfies property (P):
\begin{enumerate}
\item the abelian, \enquote{\underline{a}lmost \underline{c}yclic} group $\AC_{2^n}=C_{2^{n-1}}\times C_2$ for $n\geq3$;
\item the dihedral group
\[
\D_{2^n}=\langle x,y: x^{2^{n-1}}=y^2=1, y^{-1}xy=x^{-1}\rangle
\]
for $n\geq3$;
\item the generalized quaternion (or, synonymously, dicyclic) group
\[
\Q_{2^n}=\langle x,y: x^{2^{n-1}}=1, y^2=x^{2^{n-2}}, y^{-1}xy=x^{-1}\rangle
\]
for $n\geq4$;
\item the semidihedral group
\[
\SD_{2^n}=\langle x,y: x^{2^{n-1}}=y^2=1, y^{-1}xy=x^{2^{n-2}-1}\rangle
\]
for $n\geq4$;
\item the modular group
\[
\M_{2^n}=\langle x,y:x^{2^{n-1}}=y^2=1, y^{-1}xy=x^{2^{n-2}+1}\rangle
\]
for $n\geq4$.
\end{enumerate}
\end{proposition}

We remark that the group $\AC_4\cong\D_4\cong\Q_4\cong\SD_4\cong\M_4\cong C_2^2$ only has even complete mappings (see Theorem \ref{mainTheo1}), as does $\Q_8$ (as a simple brute-force search using GAP \cite{GAP4} shows). Moreover, $\SD_8\cong\AC_8$ and $\M_8\cong\D_8$. For $n\geq4$, the five groups $\AC_{2^n}$, $\D_{2^n}$, $\Q_{2^n}$, $\SD_{2^n}$ and $\M_{2^n}$ are pairwise non-isomorphic and are the only noncyclic groups of order $2^n$ with a cyclic subgroup of index $2$ (see \cite[result 5.3.4 on p.~141]{Rob96a}), a fact that will become important later.

\begin{proof}[Proof of Proposition \ref{specialCasesProp}]
We begin this proof by introducing some concepts studied in the literature and a few facts about them. Let $G$ be a finite group, say of order $n$. Recall from the Introduction the notion of a \emph{harmonious ordering of $G$}, that $G$ is called \emph{harmonious} if and only if it has a harmonious ordering, and the following fact:
\begin{enumerate}
\item $G$ is harmonious if and only if $G$ has a complete mapping that is an $n$-cycle. In fact, $g_1,g_2,\ldots,g_n$ is a harmonious ordering of $G$ if and only if the $n$-cycle $(g_1,g_2,\ldots,g_n)$ is a complete mapping of $G$.
\end{enumerate}
An \emph{R-sequencing of $G$}, as defined in \cite[Subsection 2.4]{Oll02a}, is a repetition-free list $1_G=g_1,g_2,\ldots,g_n$ of the elements of $G$ such that the $n-1$ partial products $b_1=g_1=1_G,b_2=g_1g_2=g_2,\ldots,b_{n-1}=g_1g_2\cdots g_{n-1}=g_2g_3\cdots g_{n-1}$ are pairwise distinct and the full product $b_n=g_1g_2\cdots g_n=g_2g_3\cdots g_n$ is equal to $1_G$. If $G$ has at least one R-sequencing, then $G$ is called \emph{R-sequenceable}. Here, the connection to complete mappings or orthomorphisms is not as obvious as for harmonious orderings, but we note:
\begin{enumerate}\setcounter{enumi}{1}
\item $G$ is R-sequenceable if and only if $G$ has an orthomorphism that fixes some nontrivial element of $G$ and moves the remaining $n-1$ elements of $G$ in a single cycle. In fact, a repetition-free listing $1_G=g_1,g_2,\ldots,g_n$ of the elements of $G$ is an R-sequencing of $G$ if and only if the partial products $b_1,b_2,\ldots,b_{n-1}$ are pairwise distinct and $(b_1,b_2,\ldots,b_{n-1})$ is an orthomorphism of $G$. This observation appears to date back to Paige's paper \cite{Pai51a}. More precisely, Paige observed that if $g_1,g_2,\ldots,g_n$ is an R-sequencing of $G$, then the function $f:G\rightarrow G$ with
\begin{itemize}
\item $f(b_i)=b_i^{-1}b_{i+1}=g_{i+1}$ for $i=2,3,\ldots,n-1$,
\item $f(b_1)=f(1_G)=g_2=b_2$, and
\item $f(c)=1_G$ where $c$ is defined via $\{c\}=G\setminus\{b_1,b_2,\ldots,b_{n-1}\}$,
\end{itemize}
is a complete mapping of $G$. Note that $\tilde{f}=(b_1,b_2,\ldots,b_{n-1})$, which proves one direction in the asserted equivalence. Conversely, if the $(n-1)$-cycle $(b_1,b_2,\ldots,b_{n-1})$ is an orthomorphism of $G$ whose unique fixed point is nontrivial, and if, say, $b_1=1_G$, then setting $g_1:=1_G$, $g_2:=b_2$, $g_j:=b_{j-1}^{-1}b_j$ for $j=3,4,\ldots,n-1$, and $g_n:=b_{n-1}^{-1}$, one verifies easily that $g_1,g_2,\ldots,g_n$ is an R-sequencing of $G$.
\end{enumerate}
We also note that some authors, such as Friedlander, Gordon and Miller in \cite{FGM78a}, instead call a repetition-free list $g_1,g_2,\ldots,g_{n-1}$ of the nontrivial elements of $G$ an R-sequencing of $G$ if $(g_1,g_2,\ldots,g_{n-1})$ is an orthomorphism of $G$. That the existence of an orthomorphism of $G$ which fixes $1_G$ and moves the nontrivial elements of $G$ in one cycle (i.e., of an \emph{$(|G|-1,1)$ complete mapping of $G$} in the terminology of \cite{HK84a}) is equivalent to $G$ being R-sequenceable was observed by Hsu and Keedwell, see \cite[Theorem 2.1(iii)]{HK84a}. The final fact which we will need is the following, and it was already noted in our Introduction:
\begin{enumerate}\setcounter{enumi}{2}
\item Let $\inv:G\rightarrow G, g\mapsto g^{-1}$, be the inversion function of $G$. A function $f:G\rightarrow G$ is a complete mapping of $G$ if and only if $f\circ\inv$ is an orthomorphism of $G$. Moreover, if $G$ is abelian, then the analogous statement with $\inv\circ f$ (more commonly written $-f$ if $G$ is written additively) in place of $f\circ\inv$ also holds.
\end{enumerate}
After these theoretical preparations, let us now turn to the actual proof of Proposition \ref{specialCasesProp}. The more technical aspects of the proof are deferred to the various Appendices.
\begin{itemize}
\item For statement (1): By \cite[Theorem 6.6]{BGHJ91a}, the group $\AC_{2^n}$ is harmonious; in particular, it has an odd complete mapping (one that is a single $2^n$-cycle) by fact (1) above. By \cite[Theorem 7]{FGM78a}, $\AC_{2^n}$ is also R-sequenceable and thus has an even orthomorphism $g$ (one of cycle type $x_1x_{2^n-1}$) by fact (2). But the inversion function $\inv$ of $\AC_{2^n}$ is even -- there are exactly $2\cdot2=4$ elements of order dividing $2$ in $\AC_{2^n}$, so $\inv$ has $4$ fixed points and $\frac{2^n-4}{2}=2^{n-1}-2$ transpositions. Consequently, $-g=\inv\circ g$ is an even complete mapping of $\AC_{2^n}$ by fact (3).
\item For statement (2): By \cite[Theorem 5.8]{BGHJ91a}, the group $\D_{2^n}$ is harmonious, so it has an odd complete mapping by fact (1). The elements of $\D_{2^n}$ can be written in normal form as $x^{\ell}y^{\epsilon}$ where $\ell\in\{0,1,\ldots,2^{n-1}-1\}$ and $\epsilon\in\{0,1\}$. Based on this representation, and setting $m:=2^{n-2}$, Hall and Paige defined a complete mapping $f$ of $\D_{2^n}$ as follows in \cite[proof of Lemma 1]{HP55a}:
\begin{equation}\label{hpEvenEq}
f(x^{\ell}y^{\epsilon})=
\begin{cases}
x^{\ell}, & \text{if }\epsilon=0\text{ and }0\leq\ell\leq m-1, \\
x^{\ell-m}y, & \text{if }\epsilon=0\text{ and }m\leq\ell\leq 2m-1, \\
x^{-(\ell+1)}, & \text{if }\epsilon=1\text{ and }0\leq\ell\leq m-1, \\
x^{m-(\ell+1)}y, & \text{if }\epsilon=1\text{ and }m\leq\ell\leq 2m-1.
\end{cases}
\end{equation}
Set
\[
B_i:=
\begin{cases}
\{x^{\ell}: 0\leq\ell\leq m-1\}, & \text{if }i=1, \\
\{x^{\ell}: m\leq\ell\leq 2m-1\}, & \text{if }i=2, \\
\{x^{\ell}y: 0\leq\ell\leq m-1\}, & \text{if }i=3, \\
\{x^{\ell}y: m\leq\ell\leq 2m-1\}, & \text{if }i=4.
\end{cases}
\]
Then it is readily checked that
\[
f(B_1)=B_1, f(B_2)=B_3, f(B_3)=B_2, f(B_4)=B_4.
\]
Each point in $B_1$ is a fixed point of $f$, and on $B_4$, all points lie on transpositions of $f$: For $m\leq\ell\leq 2m-1$, observe that
\[
x^{\ell}y\xmapsto{f}x^{m-(\ell+1)}y\xmapsto{f}x^{m-(m-(\ell+1)+1)}y=x^{\ell}y,
\]
and since $2\ell\not\equiv m-1\Mod{2m}$ (as $2\ell$ is even but $m-1$ is odd), one has $\ell\not\equiv m-(\ell+1)\Mod{2m}$, and thus $x^{\ell}y\not=x^{m-(\ell+1)}y$. Finally, the points in $B_2$ (and thus also the points in $B_3$) all lie on $4$-cycles of $f$: For $m\leq\ell\leq 2m-1$, one has
\[
x^{\ell}\xmapsto{f} x^{\ell-m}y\xmapsto{f} x^{-(\ell-m+1)}=x^{-\ell+m-1}\xmapsto{f}x^{-\ell-1}y\xmapsto{f}x^{-(-\ell-1+1)}=x^{\ell},
\]
and $x^{-\ell+m-1}=x^{m-(\ell+1)}\not=x^{\ell}$, as was argued a few lines above.

In summary, $f$ consists of $m$ fixed points, $\frac{m}{2}$ transpositions and $\frac{m}{2}$ cycles of length $4$. In particular, since $m=2^{n-2}$ is even, $f$ is an even permutation, as required.
\item For statement (3): By \cite[Theorem 1]{WL00a}, the group $\Q_{2^n}$ is harmonious (please mind the difference in notation: Wang and Leonard write $\Q_m$, not $\Q_{2m}$, for the dicyclic group of order $2m$), whence it has an odd complete mapping. Like for $\D_{2^n}$, the elements of $\Q_{2^n}$ can be written in the normal form $x^{\ell}y^{\epsilon}$ with $\ell\in\{0,1,\ldots,2^{n-1}-1\}$ and $\epsilon\in\{0,1\}$, and by \cite[proof of Lemma 1]{HP55a}, the function $f:\Q_{2^n}\rightarrow\Q_{2^n}$ with the same definition as in formula (\ref{hpEvenEq}) is a complete mapping of $\Q_{2^n}$. A verbatim argument to the one for dihedral groups above shows that $f$ is also an even permutation of $\Q_{2^n}$, as required.
\item For statement (4): Again, write the elements of $\SD_{2^n}$ in normal form as $x^{\ell}y^{\epsilon}$ with $\ell\in\{0,1,\ldots,2^{n-1}-1\}$ and $\epsilon\in\{0,1\}$. The function $f:\SD_{2^n}\rightarrow\SD_{2^n}$ with the same definition as in formula (\ref{hpEvenEq}) is a complete mapping of $\SD_{2^n}$ by \cite[proof of Lemma 1]{HP55a}, and it is an even permutation of $\SD_{2^n}$ by the same argument as for dihedral groups.

To see that $\SD_{2^n}$ also has an odd complete mapping, set $g:=\tilde{f}:z\mapsto zf(z)$, the associated orthomorphism of $f$. In Appendix A, we check that $g$ is an even permutation of $\SD_{2^n}$. But the inversion function of $\SD_{2^n}$ is odd; its fixed points are precisely those group elements $x^{\ell}y^{\epsilon}$ where
\begin{itemize}
\item $\epsilon=0$ and $\ell\in\{0,2^{n-2}\}$, or
\item $\epsilon=1$ and $2\mid\ell$,
\end{itemize}
whence $\inv$ has exactly $\frac{2^n-(2^{n-2}+2)}{2}=3\cdot 2^{n-3}-1$ transpositions. It follows by fact (3) that $g\circ\inv$ is an odd complete mapping of $\SD_{2^n}$, as required.
\item For statement (5): We can show that $\M_{2^n}$ for $n\geq4$ is harmonious (and thus has an odd complete mapping) -- see Appendix B. In Appendix C, we prove that for each even positive integer $k$, a certain function $f:\M_{16k}\rightarrow\M_{16k}$, defined in Table \ref{modularTable}, is an even complete mapping of $\M_{16k}$. In particular, $\M_{2^n}$ has an even complete mapping for all $n\geq5$. In order to see that $\M_{16}$ also has an even complete mapping, we refer to Appendix D, where the parities of complete mappings of groups of order $16$ are discussed.
\end{itemize}
\end{proof}

\begin{proposition}\label{orders16and32Prop}
All noncyclic finite groups of order $16$ or $32$ satisfy property (P).
\end{proposition}

\begin{proof}
First, observe that the two order $8$ groups $\AC_8=C_4\times C_2$ and $\D_8$ satisfy property (P) -- they are covered in the proof of Proposition \ref{specialCasesProp}(1,2). As for noncyclic groups of order $16$ (not already covered by previous arguments), we verified that they all satisfy property (P) with a random search algorithm implemented in GAP \cite{GAP4}, see Appendix D.

Finite groups of certain, in particular all \enquote{small} orders, are implemented in GAP \cite{GAP4} through the Small Groups Library constructed by Besche, Eick and O'Brien; see their paper \cite{BEO02a} for a survey of the history of classifying finite groups of particular orders. The groups of order $32$ are listed in this library as $\SmallGroup(32,i)$ where $i\in\{1,2,\ldots,51\}$, with $C_{32}=\SmallGroup(32,1)$.

Using GAP, we verified that if $G=\SmallGroup(32,i)$ is a noncyclic finite group of order $32$ with $i\notin\{2,6,16,17,18,19,20,51\}$, then $G$ has a normal subgroup $N$ isomorphic to $C_4\times C_2$ or $\D_8$ such that $G/N\cong C_2^2$, whence $G$ satisfies property (P) by Lemma \ref{mainLem}.

With regard to the remaining eight groups $G=\SmallGroup(32,i)$, we note the following:
\begin{itemize}
\item For $i=51$, we have $G\cong C_2^5$, which satisfies property (P) because it has an odd complete mapping by \cite[table on pages 82--91]{Sch16a}, while any Singer cycle of it is an even complete mapping.
\item For $i=16,17,18,19,20$, one has $G\cong\AC_{32},\M_{32},\D_{32},\SD_{32},\Q_{32}$ respectively, and these are all covered by Proposition \ref{specialCasesProp}.
\item For $i=2$, one has
\begin{align*}
G &\cong C_4\ltimes (C_4\times C_2) \\
&=\langle x,y,z: x^4=y^4=z^2=y^{-1}z^{-1}yz=1, x^{-1}yx=yz, x^{-1}yx=y\rangle.
\end{align*}
Note that the generator $x$ of the outer $C_4$ acts on the normal subgroup $\langle y,z\rangle\cong C_4\times C_2$ by an automorphism $\alpha$ of order $2$, not $4$ (taking a semidirect product with an order $4$ automorphism actually leads to another representation of the last exceptional group, with $i=6$, which will be discussed in the next bullet point). Set $H:=\langle x,y^2\rangle$, another subgroup of $G$ isomorphic to $C_4\times C_2$. Moreover, set $U:=\{1,y,z,yz\}$. Since the elements of $G$ can be written as $x^ay^bz^c$ with $a,b\in\{0,1,2,3\}$ and $c\in\{0,1\}$, we can write $G=H\cdot U$ as a product set. Moreover, observing that
\[
U^{\alpha}=\{1^{\alpha},y^{\alpha},z^{\alpha},(yz)^{\alpha}\}=\{1,yz,z,y\}=U,
\]
we find that
\[
U\cdot H=U\cdot\langle x\rangle\cdot\langle y^2\rangle=\langle x\rangle\cdot U\cdot\langle y^2\rangle=\langle x\rangle\cdot\langle y^2\rangle\cdot U=H\cdot U=G,
\]
whence $U$ is a left and right transversal of $H$ in $G$. Now, consider the permutation $\Theta$ of $U$ with
\begin{equation}\label{ThetaEq}
\Theta(1)=1,\Theta(y)=z,\Theta(z)=yz,\Theta(yz)=y.
\end{equation}
Since
\[
1\Theta(1)H=H, y\Theta(y)H=yzH, z\Theta(z)H=yz^2H=yH, yz\Theta(yz)H=y^2zH=zH,
\]
and since $H\cong C_4\times C_2=\AC_8$ satisfies property (P), we infer from Lemma \ref{mainLemVar} that $G$ satisfies property (P), as we needed to show.
\item For $i=6$, one has
\begin{align*}
G &\cong C_4\ltimes C_2^3 \\
&=\langle x,y,z,t: x^4=y^2=z^2=t^2=y^{-1}z^{-1}yz=y^{-1}t^{-1}yt=z^{-1}t^{-1}zt=1, \\
&x^{-1}yx=z, x^{-1}zx=t, x^{-1}tx=yzt\rangle.
\end{align*}
Set $H:=\langle x,yt\rangle$, and note that $H$ is a subgroup of $G$ isomorphic to $C_4\times C_2$. Moreover, set $K:=\langle y,z\rangle$, which is a subgroup of $G$ isomorphic to $C_2^2$. It is not hard to see that $G=H\cdot K$, and since $H$ and $K$ are closed under taking inverses, we also get
\[
G=G^{-1}=(H\cdot K)^{-1}=K^{-1}\cdot H^{-1}=K\cdot H.
\]
Hence $H$ and $K$ are left and right transversals of each other. The permutation $\Theta$ of $K=\{1,y,z,yz\}$ given by the same formulas as in (\ref{ThetaEq}) has the property that $\{g\Theta(g): g\in K\}$ is a left transversal of $H$ in $G$, and we conclude that $G$ has complete mappings of both parities by the same argument as for $i=2$ (previous bullet point).
\end{itemize}
\end{proof}

\begin{proof}[Proof of Theorem \ref{mainTheo4}(2)]
We will prove, by induction on $n\geq4$, that every finite noncyclic group $S$ of order $2^n$ satisfies property (P). By Proposition \ref{orders16and32Prop}, this is true if $n\in\{4,5\}$, so we may assume that $n\geq6$. We distinguish two cases:
\begin{enumerate}
\item Case: $S$ is abelian. Write $S=\prod_{i=1}^k{C_{2^{e_i}}}$ with $e_1\geq e_2\geq\cdots\geq e_k\geq1$. Note that we may assume $e_1>1$, since otherwise, $S$ is elementary abelian (a case covered by Theorem \ref{mainTheo1}). If $k\geq3$, then $S$ has a normal subgroup
\[
N\cong C_{2^{e_1-1}}\times C_{2^{e_2-1}}\times\prod_{i=3}^k{C_{2^{e_i}}},
\]
which is noncyclic as it as has at least two nontrivial cyclic factors ($C_{2^{e_1-1}}$ and $C_{2^{e_3}}$). Moreover, $|N|=2^{n-2}\geq 2^4$, so $N$ satisfies property (P) by the induction hypothesis, and $S/N\cong C_2^2$ has a complete mapping. Lemma \ref{mainLem} lets us conclude that $S$ satisfies property (P), as required.

It remains to discuss $k=2$. If $e_2>1$, then $S$ has a normal subgroup $N\cong C_{2^{e_1-1}}\times C_{2^{e_2-1}}$, and an analogous argument as for $k\geq3$ works. And if $e_2=1$, then $S$ is isomorphic to $\AC_{2^n}$ and is covered by Proposition \ref{specialCasesProp}(1).
\item Case: $S$ is nonabelian. We follow the argument given by Hall and Paige in \cite[proof of Theorem 4]{HP55a}. We may assume that $S$ does not have a cyclic subgroup of index $2$, because otherwise, $S$ is isomorphic to one of $\D_{2^n}$, $\Q_{2^n}$, $\SD_{2^n}$ or $\M_{2^n}$ by \cite[result 5.3.4 on p.~141]{Rob96a}, and those groups are covered by Proposition \ref{specialCasesProp}(2--5). In particular, $S$ is not dicyclic and thus contains more than one element of order $2$. Let $x$ be a central element of order $2$ in $S$, and let $y\not=x$ be another (possibly also central) element of order $2$. Set $V:=\langle x,y\rangle\cong C_2^2$.

If $V$ is contained in two distinct maximal subgroups $M_1$ and $M_2$ of $S$ (which are of index $2$ in $S$ since $S$ is a finite $2$-group), then $K:=M_1\cap M_2$ is a normal subgroup of $S$ with $K\geq V$ (whence $K$ is noncyclic) and $S/K\cong C_2^2$. By the induction hypothesis, $K$ satisfies property (P), and so does $S$ by Lemma \ref{mainLem}.

Now assume that $V$ is contained in a unique maximal subgroup $M_1$ of $S$. Since $S$ is noncyclic, it has another maximal subgroup, say $M_2$, and if $M_1\cap M_2$ is noncyclic, the same argument as in the last paragraph works. We may thus assume that $M_1\cap M_2$ is cyclic. Following Hall and Paige's argument further (for which the assumption that $S$ does not contain a cyclic subgroup of index $2$ is crucial), we see that $S$ has a noncyclic index $4$ subgroup $H$ with a left and right transversal $u_1,u_2,u_3,u_4$ such that $u_1\in H$ and there are permutations $S,T\in\Sym(4)$ such that $S(1)=T(1)=1$ and $u_iu_{S(i)}H=u_{T(i)}H$ for all $i\in\{1,2,3,4\}$. Since $H$ satisfies property (P) by the induction hypothesis, so does $S$ by Lemma \ref{mainLemVar}.
\end{enumerate}
\end{proof}

As was explained at the beginning of this section, Theorem \ref{mainTheo4}(3) follows from Theorem \ref{mainTheo4}(2), and so the proof of our main results is now complete.

\section{Concluding remarks}\label{sec5}

\subsection{Orthomorphisms, the alternating group, and known results}\label{subsec5P1}

The aim of this subsection is to discuss to what extent known results could be used to show for certain dimensions $n$ that the group $P_{\comp}(\IF_2^n)=P_{\orth}(\IF_2^n)$ contains $\Alt(\IF_2^n)$. More specifically, we will argue that it is \emph{not} an obvious consequence of known results that this holds for infinitely many $n$ (as either of our Theorems \ref{mainTheo1} and \ref{mainTheo2} implies).

Since $P_{\orth}(\IF_2^n)$ is defined as the group generated by all orthomorphisms of $\IF_2^n$, in order to conclude that $\Alt(\IF_2^n)\leq P_{\orth}(\IF_2^n)$ for a given $n$, it suffices if $\Alt(\IF_2^n)$ is contained in the group generated by the round functions of some cipher over $\IF_2^n$ all of whose round functions are orthomorphisms. Now, the papers known to the authors that contain a result which implies that $\Alt(\IF_2^n)$ is contained in the group generated by the round functions of a certain cipher or class of ciphers over $\IF_2^n$ are \cite{ACTT18a,ACDS14a,ACS17a,CDS09a,EG83a,HSW94a,MPW94a,SW08a,SW15a,Wer93a}. Of those papers, \cite{HSW94a,SW15a,Wer93a} deal with ciphers in specific dimensions, so we ignore them in this discussion.

A noteworthy restriction which the remaining seven papers have in common is that none of them cover prime dimensions $n$. This is due to the nature of the studied ciphers, which require one to divide the input string into several segments of equal length larger than $1$. With regard to those dimensions that are covered, we make the following observations:
\begin{itemize}
\item The papers \cite{ACTT18a} and \cite{ACDS14a} contain similar results for so-called translation based (tb) ciphers over $\IF_2^n$. They distinguish between round functions occurring in different rounds and consider the group $\Gamma_h(\Cscr)$ generated by the round functions from the $h$-th round of the tb cipher $\Cscr$, which are of the form $\gamma\lambda\rho(v)$ where $\gamma=\gamma_h$ is a so-called \emph{bricklayer transformation} (a segment-wise application of permutations), $\lambda=\lambda_h$ is a linear permutation of $\IF_2^n$, and $\rho(v):x\mapsto x+v$ is the additive translation by $v\in\IF_2^n$ (here, we use our notation $\rho$ for the (right-)regular representation of $\IF_2^n$ on itself; in the notation of both \cite{ACTT18a} and \cite{ACDS14a}, $\rho(v)$ would be written $\sigma_v$). Moreover, note that the multiplication of permutations used here is function composition in the reverse order compared to $\circ$, so $\gamma\lambda\rho(v)=\rho(v)\circ\lambda\circ\gamma$ (apply $\gamma$ first, then $\lambda$, then $\rho(v)$). The results \cite[Theorems 4.6 and 4.7]{ACTT18a} and \cite[Theorem 4.5]{ACDS14a} state that if no sum of subspaces of $\IF_2^n$ corresponding to the chosen vector segmentation other than $\{0\}$ and $\IF_2^n$ is invariant under $\lambda$ (in which case $\lambda$ is called a \emph{proper mixing layer}), and if the segment-wise permutations of which $\gamma$ consists (the so-called \emph{bricks} of $\gamma$) satisfy certain cryptographic assumptions which intuitively mean that those bricks are \enquote{far away from being linear}, then $\Gamma_h(\Ccal)=\langle \gamma\lambda\rho(v): v\in\IF_2^n\rangle=\langle\gamma\lambda,\rho(\IF_2^n)\rangle$ equals $\Alt(\IF_2^n)$. If $\gamma\lambda$ is an orthomorphism of $\IF_2^n$, then so is $\gamma\lambda\rho(v)$ for any $v\in\IF_2^n$, so in order to conclude that $\Alt(\IF_2^n)\leq P_{\orth}(\IF_2^n)$, one would need to find $\gamma$ and $\lambda$ satisfying the assumptions of one of those theorems such that $\gamma\lambda$ is an orthomorphism of $\IF_2^n$. It does not appear to be obvious whether this is possible.
\item The paper \cite{ACS17a} discusses so-called \enquote{GOST-like} ciphers. Assume that $n$ is even. View $\IF_2^n$ as the direct product $\IF_2^{n/2}\times\IF_2^{n/2}$. On each copy of $\IF_2^{n/2}$, consider two different group structures:
\begin{itemize}
\item the underlying additive group structure of its vector space structure, the group operation of which is denoted by $+$, and
\item the cyclic additive group structure that stems from viewing each element of $\IF_2^{n/2}$, which is literally a bit string, as the binary representation of a number in $\{0,1,\ldots,2^{n/2}-1\}$, and adding those numbers modulo $2^{n/2}$. The corresponding group operation is written $\oplus$, and the element of $\IF_2^{n/2}$ corresponding to $k\in\{0,1,\ldots,2^{n/2}-1\}$ is denoted by $\overline{k}$.
\end{itemize}
The round functions of the GOST-like cipher over $\IF_2^n$ considered by the authors of \cite{ACS17a} are described in \cite[top of page 5]{ACS17a}. They are of the form
\[
R_{\vec{h},\vec{k}}:(v_{\ell},v_r)\mapsto(v_r\oplus\overline{k_2}\oplus\overline{h_1},((v_{\ell}\oplus\overline{k_1})+S(v_r\oplus\overline{k_2}))\oplus\overline{h_2})
\]
for each $v=(v_{\ell},v_r)\in\IF_2^n$, where the $h_i$ and $k_i$ are numbers in $\{0,1,\ldots,2^{n/2}-1\}$, and $S$ is a permutation of $\IF_2^{n/2}$ that is fixed with the cipher (i.e., it does not vary with $R_{\vec{h},\vec{k}}$) and is of the form $\gamma R_r$ where $\gamma$ is a bricklayer transformation and $R_r$ is the cyclic right rotation by $r$ bits on $\IF_2^{n/2}$. In \cite[end of Section 2]{ACS17a}, the round functions of a cipher called GOST by the authors are specified, and they correspond to the special case $k_1=h_2=0$, $k_2=k\in\{0,1,\ldots,2^{n/2}-1\}$ and $h_1=2^{n/2}-k$. It is noteworthy that this version of GOST appears to differ from the one described in \cite[Section 3]{MM14a}, and unlike the one in \cite[Section 3]{MM14a}, it is \emph{not} a Feistel cipher, as its round functions are not of the format $R_{\kappa}$ described in our Introduction. It is, however, not hard to check that the round functions of GOST as defined in \cite[end of Section 2]{ACS17a} all are orthomorphisms of $\IF_2^n$ regardless.

However, the more general round functions of the GOST-like cipher that are shown in \cite[Theorem 3.1]{ACS17a} to generate $\Alt(\IF_2^n)$ under certain assumptions are \emph{not} all orthomorphisms. In fact, if $h_1=h_2=k_1=0$ and $h_2=1$, then
\[
R_{\vec{h},\vec{k}}(v_{\ell},v_r)=(v_r\oplus \overline{1},v_{\ell}+S(v_r\oplus \overline{1})),
\]
and
\[
\widetilde{R_{\vec{h},\vec{k}}}(v_{\ell},v_r)=(v_{\ell}+(v_r\oplus \overline{1}),v_{\ell}+v_r+S(v_r\oplus \overline{1})).
\]
Observe that $\widetilde{R_{\vec{h},\vec{k}}}$ is a permutation of $\IF_2^{n/2}$ if and only if the function
\[
f_S:\IF_2^{n/2}\rightarrow\IF_2^{n/2}, v\mapsto v+(v\oplus \overline{1})+S(v\oplus \overline{1}),
\]
is injective. Indeed, if $f_S$ is injective, then the inverse function of $\widetilde{R_{\vec{h},\vec{k}}}$ is $(w_{\ell},w_r)\mapsto(w_{\ell}+(f_S^{-1}(w_{\ell}+w_r)\oplus\overline{1}),f_S^{-1}(w_{\ell}+w_r))$, and if, conversely, $\widetilde{R_{\vec{h},\vec{k}}}$ is injective, then so is the function $\IF_2^{n/2}\rightarrow\IF_2^{n/2}$, $v\mapsto\widetilde{R_{\vec{h},\vec{k}}}(v\oplus\overline{1},v)$, which implies that $f_S$ is injective. Using GAP \cite{GAP4}, we verified that in the smallest case for $n$ according to the restrictions of \cite[Theorem 3.1]{ACS17a}, namely $n=16$ (and $\delta=4$, $m=2$ in the notation of \cite[Theorem 3.1]{ACS17a}; note that our $n$ would be $2n$ in that notation), the function $f_S$ is \emph{not} injective for any of the valid choices for $S$ in \cite[Theorem 3.1]{ACS17a}. Therefore, at least for $n=16$, \cite[Theorem 3.1]{ACS17a} cannot be used to infer that $\Alt(\IF_2^n)\leq P_{\orth}(\IF_2^n)$. It would be interesting to generalize this observation to arbitrary $n$ (of the form required for \cite[Theorem 3.1]{ACS17a}).
\item The paper \cite{CDS09a} is similar to \cite{ACTT18a} and \cite{ACDS14a} in that it considers ciphers where the round functions are of the form $\gamma\lambda\rho(v)$ for a fixed bricklayer transformation $\gamma$ and proper mixing layer $\lambda$, with $v\in\IF_2^n$ variable. Unlike for a tb cipher, it is not demanded here that round functions in successive rounds are linked via a key scheduling function, but this does not affect the group generated by the round functions, which is still $\langle \gamma\lambda\rho(v): v\in\IF_2^n\rangle$ and is shown in \cite[Theorem 2]{CDS09a} to be equal to $\Alt(\IF_2^n)$ under certain cryptographic assumptions on $\gamma$ and $\lambda$. As for \cite{ACTT18a} and \cite{ACDS14a} above, it is not clear whether there exist $\gamma$ and $\lambda$ satisfying these assumptions such that additionally, $\gamma\lambda$ is an orthomorphism.
\item The two kinds of ciphers considered in \cite{EG83a}, one of which is associated with so-called \emph{DES-like functions}, the other with \emph{2-restricted DES-like functions}, both are Feistel ciphers in which not all Feistel transformations are bijective. Hence in both cases, not all round functions are orthomorphisms.
\item Not all round functions $E_k$ of the (intentionally) weak cipher constructed in \cite{MPW94a} are orthomorphisms. For example, $E_0$ is equal to the permutation $\theta$ defined in \cite[beginning of Section 2]{MPW94a}, and $\tilde{\theta}(1)=3=\tilde{\theta}(2^{n-1}+1)$.
\item At a superficial glance, the approach in \cite{SW08a} may look like another variation of \cite{ACTT18a}, \cite{ACDS14a} or \cite{CDS09a}, as it is also concerned with ciphers where the round functions are of the form $\gamma\lambda\rho(v)$ for a fixed bricklayer transformation $\gamma$ and a fixed linear permutation $\lambda$. However, the conditions imposed in \cite[Theorem 3]{SW08a} in order to ensure that the group generated by these round functions is the alternating group are of a different nature and concern that group itself, rather than being direct restrictions on $\gamma$ and $\lambda$. As for \cite{ACTT18a}, \cite{ACDS14a} and \cite{CDS09a}, it is not clear whether $\gamma$ and $\lambda$ can be chosen such that these conditions hold and, additionally, $\gamma\lambda$ is an orthomorphism. It would be interesting to check whether there are examples among the (dual) Rijndael-like functions considered in \cite[Sections 4 and 5]{SW08a}. A random search algorithm which we implemented in GAP \cite{GAP4} did not find any evidence that not all round functions of the Rijndael cipher over $\IF_{2^8}^{4\cdot 4}\cong\IF_2^{128}$ discussed in \cite[Section 5]{SW08a}, and defined in detail in \cite[Section 3.4]{DR02a}, are complete mappings, but due to the size of the domain of definition, it is impossible to carry out a comprehensive brute-force investigation with a computer.
\end{itemize}

\subsection{Parity types of orthogonal Latin squares based on complete mappings}\label{subsec5P2}

Let $X$ be a finite set of size $n$. A \emph{Latin square over $X$} is an $(n\times n)$-matrix $L=(\ell_{i,j})_{1\leq i,j\leq n}$ with entries in $X$ such that each element of $X$ occurs exactly once in each column and in each row of $X$. Two Latin squares over $X$ are called \emph{orthogonal} if in the superposition of the two, each ordered pair from $X^2$ occurs exactly once.

In \cite{Man42a}, Mann gave a useful algebraic characterization of orthogonality of Latin squares, which we will now explain. For this, assume that we fix a linear ordering of the elements of $X$, allowing us to list them as $x_1,x_2,\ldots,x_n$. Then a Latin square $L=(\ell_{i,j})_{1\leq i,j\leq n}$ over $X$ is completely described by its $n$ \emph{row permutations}, i.e., the permutations $P_1,P_2,\ldots,P_n\in\Sym(X)$ such that for $1\leq i,j\leq n$, one has $\ell_{i,j}=P_i(x_j)$. This allows us to identify $L$ with the $n$-tuple $(P_1,P_2,\ldots,P_n)\in\Sym(X)^n$. Note that not all elements of $\Sym(X)^n$ correspond to Latin squares, but they do always correspond to an $(n\times n)$-matrix over $X$ such that each element of $X$ occurs exactly once in each row. Mann calls such a matrix an \emph{$n$-sided square over $X$}, and he defines a product of $n$-sided squares over $X$ as follows: Let $L_1$ and $L_2$ be $n$-sided squares over $X$, with associated permutation tuples $(P_1,P_2,\ldots,P_n)$ and $(Q_1,Q_2,\ldots,Q_n)$. The \emph{product of $L_1$ and $L_2$}, another $n$-sided square over $X$ written $L_1L_2$, is the $n$-sided square over $X$ corresponding to the tuple $(P_1\circ Q_1,P_2\circ Q_2,\ldots,P_n\circ Q_n)$. Here is Mann's aforementioned algebraic characterization of orthogonality of Latin squares:

\begin{theoremm}\label{mannTheo}(Mann, \cite[Theorem 1 with $r=2$]{Man42a})
Two Latin squares $L_1$ and $L_2$ over $X$ are orthogonal if and only if there exists a Latin square $L_{1,2}$ over $X$ such that $L_2=L_1L_{1,2}$.
\end{theoremm}

As a consequence of this theorem, Mann showed how one can associate with each complete mapping $f$ of a finite group $G=\{g_1,g_2,\ldots,g_n\}$ of order $n$ a pair of orthogonal Latin squares $A$ and $B=B_f$ over $G$. Denote by $\lambda:G\rightarrow\Sym(G)$ the left-regular representation of $G$ on itself, so that $\lambda(g)(h)=gh$ for all $g,h\in G$.
\begin{itemize}
\item Let $A$ be the Cayley table of $G$, which is the Latin square over $G$ corresponding to the row permutation tuple $(\lambda(g_i))_{i=1,2,\ldots,n}$ because $\lambda(g_i)(g_j)=g_ig_j$.
\item Let $C=C_f$ be the Latin square over $G$ corresponding to the row permutation tuple $(\lambda(f(g_i)))_{i=1,2,\ldots,n}$. Note that as a matrix, $C=(\lambda(f(g_i))(g_j))_{1\leq i,j\leq n}=(f(g_i)g_j)_{1\leq i,j\leq n}$, and $C$ is indeed a Latin square since $f$ is a permutation of $G$.
\item Finally, let $B=B_f:=AC_f$, which is the Latin square over $G$ corresponding to the row permutation tuple $(\lambda(g_i)\circ\lambda(f(g_i)))_{i=1,2,\ldots,n}=(\lambda(g_if(g_i)))_{i=1,2,\ldots,n}=(\lambda(\tilde{f}(g_i)))_{i=1,2,\ldots,n}$. This is indeed a Latin square over $G$ because $\tilde{f}$ is a permutation of $G$.
\end{itemize}

The approach of associating certain permutation tuples with a Latin square $L$ also allows one to assign certain parities to $L$ that have been studied by combinatorialists. For a finite set $X$, denote by $\pi:\Sym(X)\rightarrow\IF_2$ the function that maps each permutation to its parity (i.e., $\pi(f)=0$ if and only if $f$ is even). The following definition introduces the three parities discussed in \cite[Introduction]{FHW18a}:

\begin{deffinition}\label{parityDef}
Let $L=(\ell_{i,j})_{1\leq i,j\leq n}$ be a Latin square over the finite set $X=\{x_1,x_2,\ldots,x_n\}$ of size $n$.
\begin{enumerate}
\item As above, for $i=1,2,\ldots,n$, we denote by $P_i$ the \emph{$i$-th row permutation of $L$}, viz., the unique permutation in $\Sym(X)$ such that $P_i(x_j)=\ell_{i,j}$ for $j=1,2,\ldots,n$. The \emph{row parity of $L$}, written $\pi_r(L)$, is the $\IF_2$-sum of the parities of the $P_i$ for $i=1,2,\ldots,n$.
\item For $j=1,2,\ldots,n$, we denote by $Q_j$ the \emph{$j$-th column permutation of $L$}, viz., the unique permutation in $\Sym(X)$ such that $Q_j(x_i)=\ell_{i,j}$ for $i=1,2,\ldots,n$. The \emph{column parity of $L$}, written $\pi_c(L)$, is the $\IF_2$-sum of the parities of the $Q_j$ for $j=1,2,\ldots,n$.
\item For $x\in X$, we denote by $R_x$ the \emph{$x$-symbol permutation of $L$}, viz. the unique permutation in $\Sym(X)$ such that for all $1\leq i,j\leq n$, one has $R_x(x_i)=x_j$ if and only if $\ell_{i,j}=x$. The \emph{symbol parity of $L$}, written $\pi_s(L)$, is the $\IF_2$-sum of the parities of the $R_x$ for all $x\in X$.
\end{enumerate}
\end{deffinition}

The three parities $\pi_r(L)$, $\pi_c(L)$ and $\pi_s(L)$ are known for the $\IF_2$-equality
\[
\pi_r(L)+\pi_c(L)+\pi_s(L)={n \choose 2},
\]
which expresses each of them in terms of the respective other two. Numerous proofs of this fundamental relation have been discovered, see \cite[paragraph after formula (1.1)]{FHW18a}. The $\IF_2$-sum $\pi_r(L)+\pi_c(L)$ is also noteworthy due to being an isotopism invariant of Latin squares for even $n$, see \cite[Introduction]{Kot12a}. The authors of \cite{FHW18a} and \cite{Kot12a} define two different notions of the \emph{parity type} of a Latin square $L$ based on $\pi_r(L)$, $\pi_c(L)$ and $\pi_s(L)$:

\begin{deffinition}\label{parityTypeDef}
Let $L=(\ell_{i,j})_{1\leq i,j\leq n}$ be a Latin square over the finite set $X=\{x_1,x_2,\ldots,x_n\}$ of size $n$.
\begin{enumerate}
\item (\cite[Introduction]{FHW18a}) The \emph{Franceti{\'c}-Herke-Wanless (FHW) parity type of $L$}, written $\pi_{\mathrm{FHW}}(L)$, is the ordered triple $(\pi_r(L),\pi_c(L),\pi_s(L))$.
\item (\cite[Definition 1]{Kot12a}) The \emph{Kotlar (K) parity type of $L$}, written $\pi_{\mathrm{K}}(L)$, is the unique pair $(k,m)\in\{0,1,\ldots,\lfloor\frac{n}{2}\rfloor\}^2$ such that exactly $k$ of the row permutations, and exactly $m$ of the column permutations, of $L$ have a common parity.
\end{enumerate}
\end{deffinition}

The Kotlar parity type $\pi_{\mathrm{K}}$ is interesting because it is an isotopism invariant of Latin squares, see \cite[Proposition 2]{Kot12a}. Our goal in this subsection is to determine for each finite group $G$ the possible values of $\pi_{\mathrm{FHW}}(L)$ and $\pi_{\mathrm{K}}(L)$ where $L\in\{A,B_f,C_f:f\text{ is a complete mapping of }G\}$. We can deal uniformly with those three families of Latin squares:

\begin{nottation}\label{lhNot}
Let $G=\{g_1,g_2,\ldots,g_n\}$ be a finite group, and let $h\in\Sym(G)$. We denote by $L_h=L_h(G)$ the Latin square $(h(g_i)g_j)_{1\leq i,j\leq n}$ over $G$.
\end{nottation}

Note that $A=L_{\id}$, $B_f=L_{\tilde{f}}$ and $C_f=L_f$ for each complete mapping $f$ of $G$. The following proposition describes the row, column and symbol permutations of Latin squares of the form $L_h(G)$:

\begin{propposition}\label{lhProp}
Let $G=\{g_1,g_2,\ldots,g_n\}$ be a finite group, and let $h\in\Sym(G)$. Denote by $\lambda$ and $\rho$ the left- and right-regular representation of $G$ on itself respectively, and let $\inv:g\mapsto g^{-1}$ be the inversion function of $G$.
\begin{enumerate}
\item For $i=1,2,\ldots,n$, the $i$-th row permutation of $L_h(G)$ is $\lambda(h(g_i))$.
\item For $j=1,2,\ldots,n$, the $j$-th column permutation of $L_h(G)$ is $\rho(g_j)\circ h$.
\item For $x\in G$, the $x$-symbol permutation of $L_h(G)$ is $\rho(x)\circ\inv\circ h$.
\end{enumerate}
\end{propposition}

\begin{proof}
The first statement holds because $\lambda(h(g_i))(g_j)=h(g_i)g_j$, the second because $(\rho(g_j)\circ h)(g_i)=\rho(g_j)(h(g_i))=h(g_i)g_j$, and the third because $h(g_i)g_j=x$ is equivalent to $g_j=h(g_i)^{-1}x=(\rho(x)\circ\inv\circ h)(g_i)$.
\end{proof}

For each element $g$ of a group $G$, the notation $\langle g\rangle$ denotes the cyclic subgroup of $G$ generated by $g$. The following elementary lemma is useful for determining the row, column and symbol parity of $L_h$ based on Proposition \ref{lhProp}:

\begin{lemmma}\label{lhLem}
Let $G$ be a finite group.
\begin{enumerate}
\item For each $g\in G$, we have
\[
\pi(\rho(g))=\pi(\lambda(g))=
\begin{cases}
1, & \text{if }\langle g\rangle\text{ contains a nontrivial Sylow 2-subgroup of }G, \\
0, & \text{otherwise}.
\end{cases}
\]
\item Denote by $\sigma(G)$ the $\IF_2$-sum of the parities $\pi(\lambda(g))$ for $g\in G$ (equivalently by statement (1), $\sigma(G)$ is the $\IF_2$-sum of the $\pi(\rho(g))$ for $g\in G$). Then
\[
\sigma(G)=
\begin{cases}
0, & \text{if }|G|\equiv 0,1,3\Mod{4}, \\
1, & \text{if }|G|\equiv 2\Mod{4}.
\end{cases}
\]
\end{enumerate}
\end{lemmma}

\begin{proof}
For statement (1): Note that both $\rho(g)$ and $\lambda(g)$ consist of $|G:\langle g\rangle|$ cycles that are all of length $|\langle g\rangle|=\ord(g)$. Hence, if $G$ is of odd order (i.e., $G$ has no nontrivial Sylow $2$-subgroups), then all cycles of $\rho(g)$ and $\lambda(g)$ are of odd length, whence $\pi(\rho(g))=\pi(\lambda(g))=0$. Now assume that $G$ has even order. If $\langle g\rangle$ does \emph{not} contain a Sylow $2$-subgroup of $G$, then each of $\lambda(g)$ and $\rho(g)$ consists of an even number of cycles of the same length, so $\pi(\lambda(g))=\pi(\rho(g))=0$ as well. And if $\langle g\rangle$ does contain a (nontrivial) Sylow $2$-subgroup of $G$, then each of $\lambda(g)$ and $\rho(g)$ consists of an odd number of cycles of the same even length, whence $\pi(\lambda(g))=\pi(\rho(g))=1$.

For statement (2): If $G$ is of odd order, then $\pi(\rho(g))=\pi(\lambda(g))=0$ for all $g\in G$ by statement (1), and the assertion is clear. If, on the other hand, $G$ is of even order, then we need to determine the parity of the number of elements $g\in G$ whose order is divisible by $|G|_2$, the full power of $2$ divisor of $|G|$. Note that these elements $g$ are characterized by being the generator of some cyclic subgroup $C=C(g)$ of $G$ of order divisible by $|G|_2$. But any two distinct such subgroups have disjoint sets of generators, so we can just add up the parities of the numbers of generators for each $C$. If $|G|\equiv0\Mod{4}$, then the number $\phi(|C|)$ of generators of each $C$ is even (since it is divisible by $\phi(|G|_2)$), whence $\sigma(G)=0$. And if $|G|\equiv2\Mod{4}$, then $\phi(|C|)$ is even if and only if $|C|>|G|_2$ (because odd prime power factors $p^k$ in $|C|$ contribute an even factor $p^{k-1}(p-1)$ to $\phi(|C|)$), so we only need to consider the contribution from the Sylow $2$-subgroups of $G$ (which are cyclic of order $2$). Each of these has exactly $1$ generator, and by the Sylow theorems, their number is odd, so we conclude that $\sigma(G)=1$.
\end{proof}

We are now ready to determine the three parities for Latin squares of the form $L_h(G)$. We formulate this result compactly by specifying the FHW parity type $\pi_{\mathrm{FHW}}(L_h)=(\pi_r(L_h),\pi_c(L_h),\pi_s(L_h))$:

\begin{theoremm}\label{fhwTheo}
Let $G$ be a finite group, and let $h\in\Sym(G)$. Then
\[
\pi_{\mathrm{FHW}}(L_h)=
\begin{cases}
(0,0,0), & \text{if }|G|\equiv0\Mod{4}, \\
(0,\pi(h),\pi(h)), & \text{if }|G|\equiv1\Mod{4}, \\
(1,1,1), & \text{if }|G|\equiv2\Mod{4}, \\
(0,\pi(h),\pi(h)+1), & \text{if }|G|\equiv3\Mod{4}.
\end{cases}
\]
\end{theoremm}

\begin{proof}
By Proposition \ref{lhProp}(1), it is clear that $\pi_r(L_h)=\sigma(G)$, so the first entry of $\pi_{\mathrm{FHW}}(L_h)$ assumes the asserted value in each case by Lemma \ref{lhLem}(2). Moreover, by Proposition \ref{lhProp}(2), we have $\pi_c(L_h)=\sigma(G)+|G|\cdot\pi(h)$, whence the second entries are clear by Lemma \ref{lhLem}(2) as well. Finally, by Proposition \ref{lhProp}(3), we have $\pi_s(L_h)=\sigma(G)+|G|\pi(\inv)+|G|\pi(h)$. Hence, if $|G|\equiv0,2\Mod{4}$, then $\pi_s(L_h)=\sigma(G)$, and so the third entry has the asserted value by Lemma \ref{lhLem}(2). And if $|G|\equiv1,3\Mod{4}$, then
\[
\pi_s(L_h)=\sigma(G)+\pi(\inv)+\pi(h)=\pi(\inv)+\pi(h).
\]
Since $|G|$ is odd, $\inv$ has exactly one fixed point and $\frac{|G|-1}{2}$ transpositions. Therefore, if $|G|\equiv1\Mod{4}$, then $\pi(\inv)=0$, and so $\pi_s(L_h)=\pi(h)$, whereas if $|G|\equiv3\Mod{4}$, then $\pi(\inv)=1$ and $\pi_s(L_h)=\pi(h)+1$.
\end{proof}

Finally, recalling from above that $A=L_{\id}$, $B_f=L_{\tilde{f}}$ and $C_f=L_f$, we obtain the following corollary:

\begin{corrollary}\label{fhwCor}
Let $G$ be a finite group. In dependency of $|G|$, the FHW parity types of the Cayley table $A$ of $G$ and of the Latin squares $B_f$ and $C_f$ for a complete mapping $f$ of $G$ are as in Table \ref{fhwTable} (note that if $|G|\equiv2\Mod{4}$, the Sylow $2$-subgroups of $G$ are cyclic, whence $G$ has no complete mappings).
\end{corrollary}

\begin{table}[h]
\begin{center}
\begin{tabular}{|c|c|c|c|}\hline
$|G|$ & $\pi_{\mathrm{FHW}}(A)$ & $\pi_{\mathrm{FHW}}(B_f)$ & $\pi_{\mathrm{FHW}}(C_f)$ \\ \hline
$\equiv0(4)$ & $(0,0,0)$ & $(0,0,0)$ & $(0,0,0)$ \\ \hline
$1$ & $(0,0,0)$ & $(0,0,0)$ & $(0,0,0)$ \\ \hline
$\geq5,\equiv1(4)$ & $(0,0,0)$ & $(0,\pi(\widetilde{f}),\pi(\widetilde{f}))$ & $(0,\pi(f),\pi(f))$ \\ \hline
$\equiv2(4)$ & $(1,1,1)$ & n/a & n/a \\ \hline
$3$ & $(0,0,1)$ & $(0,1,0)$ & $(0,0,1)$ \\ \hline
$\geq7,\equiv3(4)$ & $(0,0,1)$ & $(0,\pi(\widetilde{f}),\pi(\widetilde{f})+1)$ & $(0,\pi(f),\pi(f)+1)$ \\ \hline
\end{tabular}
\caption{FHW parity types of $A$, $B_f$ and $C_f$}
\label{fhwTable}
\end{center}
\end{table}

Corollary \ref{fhwCor} and Theorem \ref{mainTheo4}(1) imply that if $G$ is a finite group with $|G|>3$ that satisfies the Hall-Paige condition, then
\begin{align*}
&\{\pi_{\mathrm{FHW}}(B_f):f\text{ is a complete mapping of }G\} \\
&=\{\pi_{\mathrm{FHW}}(C_f):f\text{ is a complete mapping of }G\} \\
&=
\begin{cases}
\{(0,0,0)\}, & \text{if }|G|\equiv0\Mod{4}, \\
\{(0,0,0),(0,1,1)\}, & \text{if }|G|\equiv1\Mod{4}, \\
\{(0,0,1),(0,1,0)\}, & \text{if }|G|\equiv3\Mod{4},
\end{cases}
\end{align*}
so we understand the possible values of $\pi_{\mathrm{FHW}}(B_f)$ and $\pi_{\mathrm{FHW}}(C_f)$. On the other hand, understanding the possible values of the pair $(\pi_{\mathrm{FHW}}(B_f),\pi_{\mathrm{FHW}}(C_f))$ is more delicate, as it is linked to understanding the possible values of $(\pi(f),\pi(\tilde{f}))$ for complete mappings $f$ of $G$; see also Question \ref{openQues3} in the next subsection.

Our earlier results also allow us to understand the possible values of $\pi_{\mathrm{K}}(L)$ for $L=L_h$ (in particular for $L\in\{A,B_f,C_f\}$):

\begin{corrollary}\label{kCor}
Let $G$ be a finite group, and let $h\in\Sym(G)$. Then
\[
\pi_{\mathrm{K}}(L_h)=
\begin{cases}
(0,0), & \text{if }G\text{ satisfies the Hall-Paige condition}, \\
(\frac{1}{2}|G|,\frac{1}{2}|G|), & \text{otherwise}.
\end{cases}
\]
\end{corrollary}

\begin{proof}
Denote by $o(G)$ the number of $g\in G$ such that $\langle g\rangle$ contains a nontrivial (cyclic) Sylow $2$-subgroup of $G$. By Proposition \ref{lhProp}(1,2) and Lemma \ref{lhLem}(1), there are parities $p_r,p_c\in\IF_2$ such that exactly $o(G)$ of the row permutations, respectively of the column permutations, of $L_h$ have the parity $p_r$, respectively $p_c$. The result is clear once we have shown that
\[
o(G)=
\begin{cases}
0, & \text{if }G\text{ satisfies the Hall-Paige condition}, \\
\frac{1}{2}|G|, & \text{otherwise}.
\end{cases}
\]
Indeed, if $G$ satisfies the Hall-Paige condition, then $o(G)=0$ since $G$ has no nontrivial, cyclic Sylow 2-subgroup, so assume otherwise. Then $G$ does have a (nontrivial) cyclic Sylow $2$-subgroup, and it follows by Burnside's normal $p$-complement theorem that $G$ is of the form $G=S\ltimes H$ where $S$ is a Sylow $2$-subgroup of $G$ and $H$ is the unique normal Hall $2'$-subgroup of $G$. An element $g\in G$ has the property that $\langle g\rangle$ contains a Sylow $2$-subgroup of $G$ if and only if the canonical projection $G\rightarrow S$ maps $g$ to a generator of $S$. Since the number of generators of $S$ is $\phi(|S|)=\frac{1}{2}|S|$, it follows that $o(G)=\frac{1}{2}|S|\cdot|H|=\frac{1}{2}|G|$.
\end{proof}

\subsection{Open problems}\label{subsec5P3}

We conclude this paper with a discussion of related open problems for further research. In view of statements (4) and (5) of Theorem \ref{mainTheo4}, the following question is natural:

\begin{quesstion}\label{openQues1}
Is it true that $P_{\orth}(G)=\Sym(G)$ for \emph{all} large enough finite groups $G$ that satisfy the Hall-Paige condition (not only solvable ones)?
\end{quesstion}

In fact, one may ask the following even stronger open question:

\begin{quesstion}\label{openQues2}
Is it true that all large enough finite groups that satisfy the Hall-Paige condition also satisfy property (P)?
\end{quesstion}

Proposition \ref{openQuesProp} below outlines a possible approach to answering Question \ref{openQues2} in the affirmative. Following \cite[p.~88]{Rob96a}, a \emph{centerless completely reducible (CR-)group} is a direct product of nonabelian simple groups. An \emph{extension of a group $Q$ by a group $N$} is a group $G$ with a normal subgroup $N_0\cong N$ such that $G/N_0\cong Q$. If $\Ccal_1$ and $\Ccal_2$ are classes of groups closed under isomorphism, we say that a group $G$ is \emph{$\Ccal_1$-by-$\Ccal_2$} if $G$ is an extension of a group in $\Ccal_1$ by a group in $\Ccal_2$. We remark that assumption (4) in Proposition \ref{openQuesProp} is a \enquote{property (P) analogue} of Wilcox' result \cite[Proposition 11]{Wil09a} (which is based on Evans' \cite[Theorem 11]{Eva92b}).

\begin{propposition}\label{openQuesProp}
Assume that each finite group that is an extension of at least one of the following forms satisfies property (P):
\begin{enumerate}
\item (nontrivial finite centerless CR-group)-by-(finite cyclic $2$-group);
\item (nontrivial finite centerless CR-group)-by-((finite cyclic $2$-group)-by-$\IZ/3\IZ$);
\item (nontrivial finite centerless CR-group)-by-(solvable group of order at most $24$);
\item $\IZ/2\IZ$-by-(group satisfying property (P)).
\end{enumerate}
Then every finite group $G$ with $|G|>24$ that satisfies the Hall-Paige condition also satisfies property (P).
\end{propposition}

\begin{proof}
Theorem \ref{mainTheo4}(3) guarantees that this holds if $G$ is solvable, so we may assume that $G$ is nonsolvable. Set
\[
\Rad(G):=\langle N: N\unlhd G, N\text{ is solvable}\rangle,
\]
the so-called \emph{solvable radical of $G$}, which is the largest solvable normal subgroup of $G$. Note that $G/\Rad(G)$ is nonsolvable and thus satisfies the Hall-Paige condition. If $\Rad(G)$ satisfies property (P), then so does $G$ by Lemma \ref{mainLem}, so we may assume that $\Rad(G)$ does \emph{not} satisfy property (P). We distinguish two cases:
\begin{enumerate}
\item Case: $|\Rad(G)|\leq24$. Since $G/\Rad(G)$ has no nontrivial solvable normal subgroups, it follows by \cite[result 3.3.18 on p.~89]{Rob96a} that $G/\Rad(G)$ has a nontrivial normal (in fact, characteristic) subgroup $T$ that is centerless CR. Let $H\unlhd G$ be the pre-image of $T$ under the canonical projection $G\rightarrow G/\Rad(G)$. Then $H$, being an extension of $T$ by the solvable group $\Rad(G)$ with $|\Rad(G)|\leq24$, satisfies property (P) by our assumption (3). If $G/H$ satisfies the Hall-Paige condition, then $G$ satisfies property (P) by Lemma \ref{mainLem}, so assume otherwise. Then the Sylow $2$-subgroups of $G/H$ are nontrivial and cyclic, and by Burnside's normal $p$-complement theorem, we find that $G/H$ has a normal Hall $2'$-subgroup $C$. Let $K\unlhd G$ be the pre-image of $C$ under the canonical projection $G\rightarrow G/H$. Then $K$ is an extension of $C$, which satisfies the Hall-Paige condition, by $H$, which satisfies property (P), whence $K$ satisfies property (P). But $G/K$ is a cyclic $2$-group, whence some iterated applications of assumption (4) show that $G$ satisfies property (P), as required.
\item Case: $|\Rad(G)|>24$. Since we assume that $\Rad(G)$ does not satisfy property (P), Theorem \ref{mainTheo4}(3) lets us conclude that $\Rad(G)$ does not satisfy the Hall-Paige condition. By Burnside's normal $p$-complement theorem, $\Rad(G)$ has a normal Hall $2'$-subgroup $D$, and $\Rad(G)$ is (finite cyclic $2$-group)-by-$D$. Let $T$ and $H$ be as in case (1). If $|D|\leq 3$, then $H$ satisfies property (P) by our assumptions (1) and (2), and we can conclude as in case (1). Hence, assume that $|D|>3$. Then $D$ satisfies property (P) by Theorem \ref{mainTheo4}(1). Since $G/D$ is nonsolvable and thus satisfies the Hall-Paige condition, we conclude by Lemma \ref{mainLem} that $G$ satisfies property (P), as required.
\end{enumerate}
\end{proof}

At the end of the previous subsection, we mentioned that understanding the possible parity pairs $(\pi(f),\pi(\tilde{f}))$ for complete mappings $f$ of a given finite group $G$ would also lead to an understanding of the parity type pairs $(\pi_{\mathrm{FHW}}(B_f),\pi_{\mathrm{FHW}}(C_f))$ of Latin squares. This motivates the following, yet stronger open question:

\begin{quesstion}\label{openQues3}
Is it true that for every large enough finite group $G$ that satisfies the Hall-Paige condition, one has
\[
\{(\pi(f),\pi(\tilde{f})):f\text{ is a complete mapping of }G\}=\{0,1\}\times\{0,1\}?
\]
\end{quesstion}

Two other open problems, motivated by Theorem \ref{mainTheo2} and the remarks after Corollary \ref{mainCor}, are to derive explicit lower bounds on the number of complete mappings of a finite group $G$ satisfying the Hall-Paige condition that reflect the asymptotic formula from \cite[Theorem 1.2]{EMM22a}, and to try to make the \enquote{sufficiently large} in \cite[Theorem 6.9]{MP22a} explicit. More specifically, we pose the following open problems:

\begin{problemm}\label{openProb1}
Find absolute constants $c\in\left(0,\e^{-1}\right)$ and $N>0$ such that every finite group $G$ with $|G|\geq N$ that satisfies the Hall-Paige condition has at least $c|G/G'|\frac{(|G|!)^2}{|G|^{|G|}}$ complete mappings.
\end{problemm}

\begin{problemm}\label{openProb2}
Find an absolute constant $N'$ such that every finite group $G$ with $|G|\geq N'$ that satisfies the Hall-Paige condition and is not an elementary abelian $2$-group is harmonious.
\end{problemm}

The notion of R-sequenceability was equally important for our proof of Proposition \ref{specialCasesProp} as harmoniousness, and we wonder whether there is an analogue of \cite[Theorem 6.9]{MP22a} for it:

\begin{problemm}\label{openProb3}
Characterize R-sequenceability for large enough finite groups. Is it true that every large enough finite group that satisfies the Hall-Paige condition is R-sequenceable (by \cite[Theorem 3]{OT19a}, this is true for abelian groups)?
\end{problemm}

Finally, while we were able to show in our proof of Proposition \ref{specialCasesProp}(5) that modular $2$-groups $\M_{2^n}$ for $n\geq4$ are harmonious, the following questions concerning classes of groups mentioned in Proposition \ref{specialCasesProp} are open:

\begin{quesstion}\label{openQues4}
Let $n\geq4$ be a positive integer.
\begin{enumerate}
\item Is the semidihedral group $\SD_{2^n}$ harmonious?
\item Is $\SD_{2^n}$ R-sequenceable?
\item Is the modular group $\M_{2^n}$ R-sequenceable?
\end{enumerate}
\end{quesstion}

Answering parts of Question \ref{openQues4} in the affirmative could lead to a simplification of our proofs that $\SD_{2^n}$ and $\M_{2^n}$ satisfy property (P). Indeed, since the inversion function of $\SD_{2^n}$ is odd, one may replace Appendix A by an affirmative answer to one of parts (1) or (2) of Question \ref{openQues4}. Analogously, since the inversion function of $\M_{2^n}$ is even, Appendix C could be replaced by an affirmative answer to part (3). Of course, \cite[Theorem 6.9]{MP22a} implies that $\SD_{2^n}$ is harmonious for large enough $n$.

\setcounter{secnumdepth}{0}

\section{Appendix A: Semidihedral \texorpdfstring{$2$}{2}-groups have even orthomorphisms}

Let $n\geq4$ be a positive integer, set $k:=2^{n-3}$, and consider the semidihedral group
\[
\SD_{2^n}=\SD_{8k}=\langle x,y: x^{4k}=y^2=1, y^{-1}xy=x^{2k-1}\rangle.
\]
Writing the elements of this group in normal form as $x^{\ell}y^{\epsilon}$ with $\ell\in\{0,1,\ldots,4k-1\}$ and $\epsilon\in\{0,1\}$, the group product is as follows:
\[
(x^{\ell_1}y^{\epsilon_1})\cdot(x^{\ell_2}y^{\epsilon_2})=
\begin{cases}
x^{\ell_1+\ell_2}y^{\epsilon_1+\epsilon_2}, & \text{if }\epsilon_1=0, \\
x^{\ell_1-\ell_2}y^{\epsilon_1+\epsilon_2}, & \text{if }\epsilon_1=1\text{ and }2\mid\ell_2, \\
x^{\ell_1-\ell_2+2k}y^{\epsilon_1+\epsilon_2}, & \text{if }\epsilon_1=1\text{ and }2\nmid\ell_2.
\end{cases}
\]
Consider the complete mapping $f$ of $\SD_{2^n}$ from \cite[proof of Lemma 1]{HP55a}, defined through the formula
\[
f(x^{\ell}y^{\epsilon})=
\begin{cases}
x^{\ell}, & \text{if }\epsilon=0\text{ and }0\leq\ell\leq 2k-1, \\
x^{\ell-2k}y, & \text{if }\epsilon=0\text{ and }2k\leq\ell\leq 4k-1, \\
x^{-(\ell+1)}, & \text{if }\epsilon=1\text{ and }0\leq\ell\leq 2k-1, \\
x^{2k-(\ell+1)}y, & \text{if }\epsilon=1\text{ and }2k\leq\ell\leq 4k-1.
\end{cases}
\]
Using the above group product formula, it is not hard to check that the associated orthomorphism $g:=\tilde{f}:z\mapsto zf(z)$ of $f$ is given by the following formulas (which are chosen such that the $x$-exponent of the image is always in the standard range $\{0,1,\ldots,4k-1\}$):
\begin{itemize}
\item $g(x^{\ell}y^{\epsilon})=x^{2\ell}$ if $\epsilon=0$ and $0\leq\ell\leq 2k-1$ (case 1);
\item $g(x^{\ell}y^{\epsilon})=x^{2\ell-2k}y$ if $\epsilon=0$ and $2k\leq\ell\leq 3k-1$ (case 2);
\item $g(x^{\ell}y^{\epsilon})=x^{2\ell-6k}y$ if $\epsilon=0$ and $3k\leq\ell\leq 4k-1$ (case 3);
\item $g(x^{\ell}y^{\epsilon})=x^{2\ell+1}y$ if $\epsilon=1$, $0\leq\ell\leq 2k-2$ and $2\nmid\ell$ (case 4);
\item $g(x^{\ell}y^{\epsilon})=x^{2\ell+2k+1}y$ if $\epsilon=1$, $0\leq\ell\leq k-2$ and $2\mid\ell$ (case 5);
\item $g(x^{\ell}y^{\epsilon})=x^{2\ell-2k+1}y$ if $\epsilon=1$, $k\leq\ell\leq 2k-2$ and $2\mid\ell$ (case 6);
\item $g(x^{\ell}y^{\epsilon})=x^{2\ell-2k+1}$ if $\epsilon=1$, $2k+1\leq\ell\leq 3k-1$ and $2\nmid\ell$ (case 7);
\item $g(x^{\ell}y^{\epsilon})=x^{2\ell-6k+1}$ if $\epsilon=1$, $3k+1\leq\ell\leq 4k-1$ and $2\nmid\ell$ (case 8);
\item $g(x^{\ell}y^{\epsilon})=x^{2\ell-4k+1}$ if $\epsilon=1$, $2k\leq\ell\leq 4k-2$ and $2\mid\ell$ (case 9).
\end{itemize}
In order to check that $g$ is an even permutation of $\SD_{8k}$, we will use the following combinatorial result:

\begin{theorem*}\label{inversionTheo}(\cite[Lemma 3.8]{Bon12a})
Let $X$ be a finite set, and let $<$ be a strict total order of $X$. The parity of a permutation $\sigma\in\Sym(X)$ (i.e., the parity of its number of cycles of even length) is the same as the parity of the nonnegative integer
\[
I(\sigma)=I_{<}(\sigma)=|\{(a,b)\in X^2: a<b\text{ and }\sigma(b)<\sigma(a)\}|,
\]
which counts the number of so-called \emph{inversions of $\sigma$}.
\end{theorem*}

More specifically, we say that a pair $(a,b)\in X^2$ with $a<b$ is \emph{inverted by $\sigma$} (or, synonymously, is an \emph{inversion of $\sigma$}) if $\sigma(b)<\sigma(a)$ -- note that pairs $(a,b)$ with $a>b$ do \emph{not} count as inversions of $\sigma$ under any circumstances. In Table \ref{tableA}, we count the number of inversions of $g$ with respect to the strict total order $<$ of $\SD_{8k}$ defined by
\[
x^{\ell_1}y^{\epsilon_1}<x^{\ell_2}y^{\epsilon_2} :\Leftrightarrow \epsilon_1<\epsilon_2\text{, or }\epsilon_1=\epsilon_2\text{ and }\ell_1<\ell_2,
\]
distinguishing all possible case combinations for $a=x^{\ell}y^{\epsilon}$ and $b=x^{\ell'}y^{\epsilon'}$ according to the nine cases in the formulas for the $g$-values listed above. The condition in the third column together with the case conditions characterizes when $(a,b)$ is inverted. It is important to note that cases 4--9 involve a parity condition on the $x$-exponent, which affects the counting.

\begin{center}
\begin{longtable}{|c|c|c|c|}
\hline
$a$-case & $b$-case & Characterization of when $(a,b)$ is inverted & Number of inversions \\ \hline
$1$ & $1$ & never ($a<b\Rightarrow \ell<\ell'\Rightarrow g(a)<g(b)$) & $0$ \\ \hline
$1$ & $2$--$6$ & never ($g(b)>g(a)$ throughout) & $0$ \\ \hline
$1$ & $7$ & $2k+1\leq\ell'\leq \ell+k-1$ & $\frac{1}{4}k(k-2)$ \\ \hline
$1$ & $8$ & $3k+1\leq\ell'\leq\min\{\ell+3k-1,4k-1\}$ & $\frac{1}{4}k(3k-2)$ \\ \hline
$1$ & $9$ & $2k\leq\ell'\leq\ell+2k-1$ & $k^2$ \\ \hline
$2$ & $1$ & never ($b<a$ throughout) & $0$ \\ \hline
$2$ & $2$ & never ($a<b\Rightarrow \ell<\ell'\Rightarrow g(a)<g(b)$) & $0$ \\ \hline
$2$ & $3$ & $3k\leq\ell'\leq 4k-1$ & $k^2$ \\ \hline
$2$ & $4$ & $0\leq\ell'\leq\ell-k-1$ & $\frac{1}{4}k(3k-2)$ \\ \hline
$2$ & $5$ & $0\leq\ell'\leq\ell-2k-1$ & $\frac{1}{4}k^2$ \\ \hline
$2$ & $6$ & $k\leq\ell'\leq 2k-1$ & $\frac{1}{2}k^2$ \\ \hline
$2$ & $7$ & $2k\leq\ell'\leq 3k-1$ & $\frac{1}{2}k^2$ \\ \hline
$2$ & $8$ & $3k\leq\ell'\leq 4k-1$ & $\frac{1}{2}k^2$ \\ \hline
$2$ & $9$ & $2k\leq\ell'\leq 4k-1$ & $k^2$ \\ \hline
$3$ & $1$--$2$ & never ($b<a$ throughout) & $0$ \\ \hline
$3$ & $3$ & never ($a<b\Rightarrow \ell<\ell'\Rightarrow g(a)<g(b)$) & $0$ \\ \hline
$3$ & $4$ & $0\leq\ell'\leq\ell-3k-1$ & $\frac{1}{4}k(k-2)$ \\ \hline
$3$ & $5$ & never ($g(b)>g(a)$ throughout) & $0$ \\ \hline
$3$ & $6$ & $k\leq\ell'\leq\ell-2k-1$ & $\frac{1}{4}k^2$ \\ \hline
$3$ & $7$ & $2k\leq\ell'\leq 3k-1$ & $\frac{1}{2}k^2$ \\ \hline
$3$ & $8$ & $3k\leq\ell'\leq 4k-1$ & $\frac{1}{2}k^2$ \\ \hline
$3$ & $9$ & $2k\leq\ell'\leq 4k-1$ & $k^2$ \\ \hline
$4$ & $1$--$3$ & never ($b<a$ throughout) & $0$ \\ \hline
$4$ & $4$ & never ($a<b\Rightarrow \ell<\ell'\Rightarrow g(a)<g(b)$) & $0$ \\ \hline
$4$ & $5$ & never ($a<b\Rightarrow \ell<\ell'\Rightarrow g(a)<g(b)$) & $0$ \\ \hline
$4$ & $6$ & $\max\{k,\ell+1\}\leq\ell'\leq\min\{2k-1,\ell+k-1\}$ & $\frac{1}{4}k^2$ \\ \hline
$4$ & $7$--$9$ & always & $2k^2$ \\ \hline
$5$ & $1$--$3$ & never ($b<a$ throughout) & $0$ \\ \hline
$5$ & $4$ & $\ell+1\leq\ell'\leq\ell+k-1$ & $\frac{1}{4}k^2$ \\ \hline
$5$ & $5$ & never ($a<b\Rightarrow \ell<\ell'\Rightarrow g(a)<g(b)$) & $0$ \\ \hline
$5$ & $6$--$9$ & always & $\frac{5}{4}k^2$ \\ \hline
$6$ & $1$--$3$ & never ($b<a$ throughout) & $0$ \\ \hline
$6$ & $4$ & never ($a<b\Rightarrow \ell<\ell'\Rightarrow g(a)<g(b)$) & $0$ \\ \hline
$6$ & $5$ & never ($b<a$ throughout) & $0$ \\ \hline
$6$ & $6$ & never ($a<b\Rightarrow \ell<\ell'\Rightarrow g(a)<g(b)$) & $0$ \\ \hline
$6$ & $7$--$9$ & always & $k^2$ \\ \hline
$7$ & $1$ & never ($b<a$ throughout) & $0$ \\ \hline
$7$ & $2$--$6$ & never ($g(b)>g(a)$ throughout) & $0$ \\ \hline
$7$ & $7$ & never ($a<b\Rightarrow \ell<\ell'\Rightarrow g(a)<g(b)$) & $0$ \\ \hline
$7$ & $8$ & always & $\frac{1}{4}k^2$ \\ \hline
$7$ & $9$ & $\ell+1\leq\ell'\leq\ell+k-1$ & $\frac{1}{4}k^2$ \\ \hline
$8$ & $1$--$7$ & never ($b<a$ throughout) & $0$ \\ \hline
$8$ & $8$--$9$ & never ($a<b\Rightarrow \ell<\ell'\Rightarrow g(a)<g(b)$) & $0$ \\ \hline
$9$ & $1$--$6$ & never ($b<a$ throughout) & $0$ \\ \hline
$9$ & $7$ & never ($a<b\Rightarrow \ell<\ell'\Rightarrow g(a)<g(b)$) & $0$ \\ \hline
$9$ & $8$ & $\max\{3k,\ell+1\}\leq\ell'\leq\min\{4k-1,\ell+k-1\}$ & $\frac{1}{4}k^2$ \\ \hline
$9$ & $9$ & never ($a<b\Rightarrow \ell<\ell'\Rightarrow g(a)<g(b)$) & $0$ \\ \hline
\caption{Counting the inversions of $g:\SD_{2^n}\rightarrow\SD_{2^n}$.}
\label{tableA}
\end{longtable}
\end{center}

In summary, we obtain that
\[
I(g)=2\cdot\frac{1}{4}k(k-2)+2\cdot\frac{1}{4}k(3k-2)+\frac{25}{2}k^2=\frac{29}{2}k^2-2k.
\]
In particular, $I(g)$ is even, as required.

\section{Appendix B: Modular \texorpdfstring{$2$}{2}-groups are harmonious}

Let $n\geq4$ be a positive integer. Set $m:=2^{n-2}$, and consider the modular group of order $2^n=4m$,
\[
\M_{2^n}=\M_{4m}=\langle x,y: x^{2m}=y^2=1, y^{-1}xy=x^{m+1}\rangle.
\]
Write the elements of this group in normal form as $x^{\ell}y^{\epsilon}$ with $\ell\in\{0,1,\ldots,2m-1\}$ and $\epsilon\in\{0,1\}$. We want to show that $\M_{2^n}$ is harmonious, i.e., that it has a complete mapping that moves all $2^n$ group elements in a single cycle. Consider the following function $f:\M_{2^n}\rightarrow\M_{2^n}$, which is a slight modification of the complete mapping of $\M_{2^n}$ specified by Hall and Paige in \cite[proof of Lemma 1]{HP55a}:
\[
f(x^{\ell}y^{\epsilon})=
\begin{cases}
x^{\ell+1}, & \text{if }\epsilon=0\text{ and }1\leq\ell\leq m, \\
x^{\ell+m+1}y, & \text{if }\epsilon=0\text{, and either }\ell=0\text{ or }m+1\leq\ell\leq 2m-1, \\
x^{\ell+m+2}, & \text{if }\epsilon=1\text{ and }0\leq\ell\leq m-1, \\
x^{\ell+2}y, & \text{if }\epsilon=1\text{ and }m\leq\ell\leq 2m-1.
\end{cases}
\]
We will check that $f$ is a complete mapping of $\M_{2^n}$ consisting of a single cycle. To verify that $f$ is a $2^n$-cycle, distinguish two cases according to the parity of $n$:
\begin{itemize}
\item If $n$ is even, then $m=2^{n-2}\equiv1\Mod{3}$. When spelling out the cycle, we use the notation
\[
(g_1(i)\mapsto g_2(i)\mapsto\cdots\mapsto g_r(i):i=i_0,i_0+1,\ldots,i_1)
\]
to denote a segment of the cycle of the form
\begin{align*}
&g_1(i_0)\mapsto g_2(i_0)\mapsto\cdots\mapsto g_r(i_0) \\
&\mapsto g_1(i_0+1)\mapsto g_2(i_0+1)\mapsto\cdots\mapsto g_r(i_0+1) \\
&\mapsto\cdots \\
&\mapsto g_1(i_1)\mapsto g_2(i_1)\mapsto\cdots\mapsto g_r(i_1).
\end{align*}
The tip of the arrow pointing at this segment is glued to $g_1(i_0)$ and the shaft of the arrow departing from the segment is glued to $g_r(i_1)$. We note that $r$ may be $1$. With this notational convention, the cycle of $f$ looks as follows:
\begin{align*}
1&\mapsto (x^{m+2t-1}y: t=1,2,\ldots,\frac{m}{2})\mapsto(x^{3u-2}y\mapsto x^{m+3u}: u=1,2,\ldots,\frac{m-1}{3}) \\
&\mapsto(x^{m+2s}y: s=0,1,\ldots,\frac{m}{2}-1)\mapsto(x^{3v}y\mapsto x^{m+2+3v}: v=0,1,\ldots,\frac{m-4}{3}) \\
&\mapsto x^{m-1}y\mapsto(x^w: w=1,2,\ldots,m) \\
&\mapsto(x^{m+1+3z}\mapsto x^{2+3z}y: z=0,1,\ldots,\frac{m-4}{3})\mapsto1.
\end{align*}
This is a cycle of length $4m=|\M_{2^n}|$.
\item If $n$ is odd, then $m=2^{n-2}\equiv2\Mod{3}$, and the cycle of $f$ looks as follows:
\begin{align*}
1&\mapsto (x^{m+2t-1}y: t=1,2,\ldots,\frac{m}{2})\mapsto(x^{3u-2}y\mapsto x^{m+3u}: u=1,2,\ldots,\frac{m-2}{3}) \\
&\mapsto x^{m-1}y\mapsto(x^w: w=1,2,\ldots,m) \\
&\mapsto(x^{m+1+3z}\mapsto x^{2+3z}y: z=0,1,\ldots,\frac{m-2}{3})\mapsto(x^{m+2s}y: s=1,\ldots,\frac{m}{2}-1) \\
&\mapsto(x^{3v}y\mapsto x^{m+2+3v}: v=0,1,\ldots,\frac{m-5}{3})\mapsto x^{m-2}y\mapsto 1. 
\end{align*}
Again, this is a cycle of length $4m$, as required.
\end{itemize}
It remains to check that $\tilde{f}:g\mapsto gf(g)$, is a permutation of $\M_{2^n}$. It is given by the following formulas:
\[
\tilde{f}(x^{\ell}y^{\epsilon})=
\begin{cases}
x^{2\ell+1}, & \text{if }\epsilon=0\text{ and }1\leq\ell\leq m, \\
x^{2\ell+m+1}y, & \text{if }\epsilon=0\text{, and either }\ell=0\text{ or }m+1\leq\ell\leq 2m-1, \\
x^{(m+2)\ell+m+2}y, & \text{if }\epsilon=1\text{ and }0\leq\ell\leq m-1, \\
x^{(m+2)\ell+2}, & \text{if }\epsilon=1\text{ and }m\leq\ell\leq 2m-1.
\end{cases}
\]
Each of the two functions
\[
\IZ/2m\IZ\rightarrow\IZ/2m\IZ, a\mapsto 2a\text{ resp. }a\mapsto(m+2)a=2(2^{n-3}+1)a,
\]
is a group endomorphism of $\IZ/2m\IZ$ that assumes its full image, $2\IZ/2m\IZ$, on any consecutive interval of $m$ arguments, $a_0+2m\IZ,a_0+1+2m\IZ,\ldots,a_0+m+2m\IZ$. Therefore,
\begin{itemize}
\item the images in the first case in the formula for $\tilde{f}$ are just $x,x^3,\ldots,x^{2m-1}$;
\item the images in the second case are just $xy,x^3y,\ldots,x^{2m-1}y$;
\item the images in the third case are just $y,x^2y,\ldots,x^{2m-2}y$;
\item the images in the fourth case are just $1,x^2,\ldots,x^{2m-2}$.
\end{itemize}
Hence, $\tilde{f}$ is surjective onto $\M_{2^n}$ and thus is a indeed a permutation of $\M_{2^n}$, as we wanted to show.

\section{Appendix C: Modular \texorpdfstring{$2$}{2}-groups have even complete mappings}

Let $k$ be an even positive integer. We will show that the modular group of order $16k$,
\[
\M_{16k}=\langle x,y:x^{8k}=y^2=1, y^{-1}xy=x^{1+4k}\rangle,
\]
has an even complete mapping. With $k=2^{n-4}$ for $n\geq5$, this together with Appendix B covers the statement of Proposition \ref{specialCasesProp}(5) except for $\M_{16}$, which will be dealt with in Appendix D.

Unlike for our treatment of semidihedral groups in Appendix A, taking $\tilde{f}\circ\inv$ for the complete mapping $f$ specified by Hall and Paige in \cite[proof of Lemma 1]{HP55a} always results in an odd complete mapping of $\M_{16k}$ (as we checked by hand with an unpublished theoretical argument), and also modifying Hall and Paige's $f$ by multiplying with a constant in each of its four definition cases (as we did in our one-cycle complete mapping from Appendix B) does not seem to help -- we checked this for small values of $k$ with some GAP \cite{GAP4} computer experiments. A different approach thus appears to be necessary.

We obtained the even complete mapping $f$ specified below using Wilcox' construction from \cite[proof of Proposition 7]{Wil09a}. More specifically, set $N:=\{1,x^{4k}\}$, a (central) order $2$ normal subgroup of $\M_{16k}$. Then $\M_{16k}/N\cong\AC_{8k}$. Wilcox' construction allows us to lift a complete mapping $\overline{f}$ of the quotient group $\AC_{8k}$ to a complete mapping of $\M_{16k}$. For $\overline{f}$, we used $-g=\inv\circ g$, where $g$ is the orthomorphism of $\AC_{8k}$ of cycle type $x_1x_{8k-1}$ obtained from the R-sequencing of $\AC_{8k}$ given by Friedlander, Gordon and Miller in \cite[proof of Theorem 7]{FGM78a}.

We omit the rather lengthy computational details that lead to the derivation of $f$ -- using \cite[Proposition 7]{Wil09a}, they would serve as proof that $f$ is a complete mapping of $\M_{16k}$, but it is actually easier to check this directly. In fact, the following table both defines $f$ and shows that $f$ is a complete mapping of $\M_{16k}$. To understand this table, note that we write the elements of $\M_{16k}$ in normal form as $y^{\epsilon}x^{\ell}$ with $\epsilon\in\{0,1\}$ and $\ell\in\{0,1,\ldots,8k-1\}$ (the different order of the two powers compared to Appendix B is intentional), which leads to the group product
\[
(y^{\epsilon_1}x^{\ell_1})\cdot(y^{\epsilon_2}x^{\ell_2})=
\begin{cases}
y^{\epsilon_1+\epsilon_2}x^{\ell_1+\ell_2}, & \text{if }\epsilon_2=0\text{ or }2\mid\ell_1, \\
y^{\epsilon_1+\epsilon_2}x^{\ell_1+\ell_2+4k}, & \text{otherwise}.
\end{cases}
\]
Consider an element $g=y^{\epsilon}x^{\ell}\in\M_{16k}$ in normal form. The second column of the following table describes a case into which $g$ falls according to simple arithmetical restrictions on $\epsilon$ and $\ell$. More precisely, the case description consists of a quadruple $(\epsilon,r,\ell_0,\ell_1)$ of parameters, where $r\in\{0,1,2,3\}$ is $\ell\mod{4}$, the remainder of dividing $\ell$ by $4$, and $\ell$ ranges inclusively from $\ell_0$ to $\ell_1$. The third column then gives a uniform formula for the function value $f(g)=y^{\epsilon'}x^{\ell'}$ in that case. Some cases only concern a single element $g\in\M_{16k}$ -- then the third column just contains the normal form of the single element $f(g)$. The fourth column describes the possible pairs $(\epsilon',\ell')$ that can be achieved in that case, in a way that is analogous to the second column: The description consists of a quadruple $(\epsilon',r',\ell'_0,\ell'_1)$ such that $r'=\ell'\mod{4}$ and $\ell'_0\leq\ell'\leq\ell'_1$. The fifth column of the table contains a formula for $\tilde{f}(g)=gf(g)=y^{\epsilon''}x^{\ell''}$. Finally, the sixth column characterizes the pairs $(\epsilon'',\ell'')$ that can be achieved in that case, analogously to the second and fourth columns, but with a subtle difference: The characterization consists of a quadruple $(\epsilon'',r'',\ell''_0,\ell''_1)$, where $r'\in\{0,1,\ldots,7\}$ is the remainder of $\ell'$ upon division by $8$ (not $4$, as in the second and fourth columns), and $\ell''_0\leq\ell''\leq\ell''_1$. It is not hard to check that the cases described in the second (resp.~fourth, resp.~sixth) column partition all possibilities (in the \enquote{standard set} of $\{0,1\}\times\{0,1,\ldots,8k-1\}$) for $(\epsilon,\ell)$ (resp.~$(\epsilon',\ell')$, resp.~$(\epsilon'',\ell'')$), so that $f$ is a well-defined complete mapping of $\M_{16k}$.

\begin{center}
\begin{longtable}{|c|c|c|c|c|c|}
\hline
No. & $(\epsilon,r,\ell_0,\ell_1)$ & $y^{\epsilon'}x^{\ell'}$ & $(\epsilon',r',\ell'_0,\ell'_1)$ & $y^{\epsilon''}x^{\ell''}$ & $(\epsilon'',r'',\ell''_0,\ell''_1)$ \\ \hline
1 & $(0,0,0,0)$ & $1$ & $(0,0,0,0)$ & $1$ & $(0,0,0,0)$ \\ \hline
2 & $(0,0,4,2k)$ & $yx^{\ell-3}$ & $(1,1,1,2k-3)$ & $yx^{2\ell-3}$ & $(1,5,5,4k-3)$ \\ \hline
3 & $(0,0,2k+4,4k-4)$ & $x^{\ell-1}$ & $(0,3,2k+3,4k-5)$ & $x^{2\ell-1}$ & $(0,7,4k+7,8k-9)$ \\ \hline
4 & $(0,0,4k,4k)$ & $yx^{4k}$ & $(1,0,4k,4k)$ & $y$ & $(1,0,0,0)$ \\ \hline
5 & $(0,0,4k+4,6k)$ & $x^{\ell+4k-3}$ & $(0,1,1,2k-3)$ & $x^{2\ell+4k-3}$ & $(0,5,4k+5,8k-3)$ \\ \hline
6 & $(0,0,6k+4,8k-4)$ & $yx^{\ell+4k-1}$ & $(1,3,2k+3,4k-5)$ & $yx^{2\ell+4k-1}$ & $(1,7,7,4k-9)$ \\ \hline
7 & $(0,1,1,1)$ & $yx^{8k-1}$ & $(1,3,8k-1,8k-1)$ & $yx^{4k}$ & $(1,0,4k,4k)$ \\ \hline
8 & $(0,1,5,2k+1)$ & $x^{\ell-3}$ & $(0,2,2,2k-2)$ & $x^{2\ell-3}$ & $(0,7,7,4k-1)$ \\ \hline
9 & $(0,1,2k+5,4k-3)$ & $yx^{\ell+4k-2}$ & $(1,3,6k+3,8k-5)$ & $yx^{2\ell-2}$ & $(1,0,4k+8,8k-8)$ \\ \hline
10 & $(0,1,4k+1,4k+1)$ & $x^{8k-1}$ & $(0,3,8k-1,8k-1)$ & $x^{4k}$ & $(0,0,4k,4k)$ \\ \hline
11 & $(0,1,4k+5,6k+1)$ & $yx^{\ell-3}$ & $(1,2,4k+2,6k-2)$ & $yx^{2\ell+4k-3}$ & $(1,7,4k+7,8k-1)$ \\ \hline
12 & $(0,1,6k+5,8k-3)$ & $x^{\ell-2}$ & $(0,3,6k+3,8k-5)$ & $x^{2\ell-2}$ & $(0,0,4k+8,8k-8)$ \\ \hline
13 & $(0,2,2,2k-2)$ & $yx^{\ell}$ & $(1,2,2,2k-2)$ & $yx^{2\ell}$ & $(1,4,4,4k-4)$ \\ \hline
14 & $(0,2,2k+2,4k-2)$ & $yx^{\ell+4k-2}$ & $(1,0,6k,8k-4)$ & $yx^{2\ell+4k-2}$ & $(1,2,2,4k-6)$ \\ \hline
15 & $(0,2,4k+2,6k-2)$ & $x^{\ell}$ & $(0,2,4k+2,6k-2)$ & $x^{2\ell}$ & $(0,4,4,4k-4)$ \\ \hline
16 & $(0,2,6k+2,8k-2)$ & $x^{\ell+4k-2}$ & $(0,0,2k,4k-4)$ & $x^{2\ell+4k-2}$ & $(0,2,2,4k-6)$ \\ \hline
17 & $(0,3,3,2k-1)$ & $x^{\ell+4k}$ & $(0,3,4k+3,6k-1)$ & $x^{2\ell+4k}$ & $(0,6,4k+6,8k-2)$ \\ \hline
18 & $(0,3,2k+3,4k-1)$ & $yx^{\ell+4k-1}$ & $(1,2,6k+2,8k-2)$ & $yx^{2\ell-1}$ & $(1,5,4k+5,8k-3)$ \\ \hline
19 & $(0,3,4k+3,6k-1)$ & $yx^{\ell}$ & $(1,3,4k+3,6k-1)$ & $yx^{2\ell+4k}$ & $(1,6,4k+6,8k-2)$ \\ \hline
20 & $(0,3,6k+3,8k-1)$ & $x^{\ell+4k-1}$ & $(0,2,2k+2,4k-2)$ & $x^{2\ell+4k-1}$ & $(0,5,5,4k-3)$ \\ \hline
21 & $(1,0,0,0)$ & $yx^{4k-2}$ & $(1,2,4k-2,4k-2)$ & $x^{4k-2}$ & $(0,6,4k-2,4k-2)$ \\ \hline
22 & $(1,0,4,2k-4)$ & $x^{\ell}$ & $(0,0,4,2k-4)$ & $yx^{2\ell}$ & $(1,0,8,4k-8)$ \\ \hline
23 & $(1,0,2k,2k)$ & $x^{2k-1}$ & $(0,3,2k-1,2k-1)$ & $yx^{4k-1}$ & $(1,7,4k-1,4k-1)$ \\ \hline
24 & $(1,0,2k+4,4k)$ & $x^{\ell+4k-2}$ & $(0,2,6k+2,8k-2)$ & $yx^{2\ell+4k-2}$ & $(1,6,6,4k-2)$ \\ \hline
25 & $(1,0,4k+4,6k-4)$ & $yx^{\ell}$ & $(1,0,4k+4,6k-4)$ & $x^{2\ell}$ & $(0,0,8,4k-8)$ \\ \hline
26 & $(1,0,6k,6k)$ & $yx^{2k-1}$ & $(1,3,2k-1,2k-1)$ & $x^{8k-1}$ & $(0,7,8k-1,8k-1)$ \\ \hline
27 & $(1,0,6k+4,8k-4)$ & $yx^{\ell+4k-2}$ & $(1,2,2k+2,4k-6)$ & $x^{2\ell+4k-2}$ & $(0,6,6,4k-10)$ \\ \hline
28 & $(1,1,1,2k-3)$ & $x^{\ell+4k}$ & $(0,1,4k+1,6k-3)$ & $yx^{2\ell+4k}$ & $(1,2,4k+2,8k-6)$ \\ \hline
29 & $(1,1,2k+1,4k-3)$ & $yx^{\ell-1}$ & $(1,0,2k,4k-4)$ & $x^{2\ell+4k-1}$ & $(0,1,1,4k-7)$ \\ \hline
30 & $(1,1,4k+1,6k-3)$ & $yx^{\ell}$ & $(1,1,4k+1,6k-3)$ & $x^{2\ell+4k}$ & $(0,2,4k+2,8k-6)$ \\ \hline
31 & $(1,1,6k+1,8k-3)$ & $x^{\ell-1}$ & $(0,0,6k,8k-4)$ & $yx^{2\ell-1}$ & $(1,1,4k+1,8k-7)$ \\ \hline
32 & $(1,2,2,2)$ & $yx^{4k-1}$ & $(1,3,4k-1,4k-1)$ & $x^{4k+1}$ & $(0,1,4k+1,4k+1)$ \\ \hline
33 & $(1,2,6,2k-2)$ & $x^{\ell-3}$ & $(0,3,3,2k-5)$ & $yx^{2\ell-3}$ & $(1,1,9,4k-7)$ \\ \hline
34 & $(1,2,2k+2,4k-2)$ & $yx^{\ell-1}$ & $(1,1,2k+1,4k-3)$ & $x^{2\ell-1}$ & $(0,3,4k+3,8k-5)$ \\ \hline
35 & $(1,2,4k+2,4k+2)$ & $x^{4k-1}$ & $(0,3,4k-1,4k-1)$ & $yx$ & $(1,1,1,1)$ \\ \hline
36 & $(1,2,4k+6,6k-2)$ & $yx^{\ell+4k-3}$ & $(1,3,3,2k-5)$ & $x^{2\ell+4k-3}$ & $(0,1,4k+9,8k-7)$ \\ \hline
37 & $(1,2,6k+2,8k-2)$ & $x^{\ell+4k-1}$ & $(0,1,2k+1,4k-3)$ & $yx^{2\ell+4k-1}$ & $(1,3,3,4k-5)$ \\ \hline
38 & $(1,3,3,2k-1)$ & $x^{\ell+4k-3}$ & $(0,0,4k,6k-4)$ & $yx^{2\ell+4k-3}$ & $(1,3,4k+3,8k-5)$ \\ \hline
39 & $(1,3,2k+3,4k-1)$ & $yx^{\ell+4k-2}$ & $(1,1,6k+1,8k-3)$ & $x^{2\ell-2}$ & $(0,4,4k+4,8k-4)$ \\ \hline
40 & $(1,3,4k+3,6k-1)$ & $yx^{\ell+4k-3}$ & $(1,0,0,2k-4)$ & $x^{2\ell-3}$ & $(0,3,3,4k-5)$ \\ \hline
41 & $(1,3,6k+3,8k-1)$ & $x^{\ell-2}$ & $(0,1,6k+1,8k-3)$ & $yx^{2\ell-2}$ & $(1,4,4k+4,8k-4)$ \\ \hline
\caption{Details on our even complete mapping $f$ of $\M_{2^n}$.}
\label{modularTable}
\end{longtable}
\end{center}

It remains to show that $f$ is an even permutation. Following the approach in Appendix A, we will count the inversions of $f$ with respect to a certain strict total order of $\M_{16k}$, namely the one where
\[
y^{\epsilon_1}x^{\ell_1}<y^{\epsilon_2}x^{\ell_2} :\Leftrightarrow \epsilon_1<\epsilon_2, \text{or }\epsilon_1=\epsilon_2\text{ and }\ell_1<\ell_2.
\]
Unlike in Appendix A, there are too many cases to count the inversions by hand, though (one would need to consider $41^2=1681$ pairs of case combinations). Instead, we will give a theoretical argument for why the inversions of $f$ are counted by a quadratic polynomial in $k$, which leads to the following key result:

\begin{proposition*}
For all even $k\geq2$, the number $I(f)$ of inversions of $f$ with respect to the above total order of $\M_{16k}$ is given by $59k^2+19k-6$. In particular, $I(f)$ is even, whence $f$ is an even permutation of $\M_{16k}$.
\end{proposition*}

\begin{proof}
For $2\leq k\leq18$, we verified this using GAP \cite{GAP4}. It suffices to show that for $k\geq14$, one has $I(f)=c_2k^2+c_1k+c_0$ for some rational constants $c_i$ (independent of $k$) -- then $I(f)=59k^2+19k-6$ can be inferred using interpolation. Let $a,b\in\M_{16k}$, each corresponding to one of the $41$ cases from the definition of $f$, and write $a=y^{\epsilon}x^{\ell}$ and $b=y^{\epsilon'}x^{\ell'}$. Then the values of $\epsilon$ and $\epsilon'$ are constant (depending on the case), and
\begin{equation}\label{ellEq}
\alpha_1k+\beta_1\leq\ell\leq\alpha_2k+\beta_2, \ell\equiv\beta_1\equiv\beta_2\Mod{4}
\end{equation}
for some $\alpha_i\in\{0,2,4,6,8\}$ with $\alpha_2\in\{\alpha_1,\alpha_1+2\}$ (the equality $\alpha_2=\alpha_1$ occurs if and only if the case for $\ell$ is singular, corresponding to a single value of $\ell$) and some integers $\beta_i$ with $|\beta_i|\leq 6$. Analogously,
\begin{equation}\label{ellPrimeEq}
\gamma_1k+\delta_1\leq\ell'\leq\gamma_2k+\delta_2, \ell'\equiv\delta_1\equiv\delta_2\Mod{4}
\end{equation}
for some $\gamma_i\in\{0,2,4,6,8\}$ with $\gamma_2\in\{\gamma_1,\gamma_1+2\}$ and some integers $\delta_i$ with $|\delta_i|\leq6$. The function values $f(a)$ and $f(b)$ are of a particular shape, namely
\[
f(a)=y^{\delta}x^{\ell+\zeta k+\eta}\text{ and }f(b)=y^{\delta'}x^{\ell'+\zeta'k+\eta'}
\]
for some integer constants $\delta,\delta'\in\{0,1\}$, $\zeta,\zeta'\in\{0,4\}$ and $\eta,\eta'$ with $|\eta|,|\eta'|\leq3$. We will call $\epsilon$, $\epsilon'$, $\alpha_1$, $\alpha_2$, $\beta_1$, $\beta_2$, $\gamma_1$, $\gamma_2$, $\delta_1$, $\delta_2$, $\delta$, $\delta'$, $\zeta$, $\zeta'$, $\eta$ and $\eta'$ the \emph{case parameters}, as they characterize which pair of cases for $a$ and $b$ we are considering.

It is enough to prove that the number of pairs $(a,b)$ satisfying these case-pair-specific conditions as well as $a<b$ and $f(b)<f(a)$ is given by a polynomial of the form $c'_2k^2+c'_1k+c'_0$ for some rational constants $c'_i$, which may depend on the case parameters, but not on $k$, as the total number of inversions of $f$ then is simply a sum of $41^2$ such polynomials. In the remainder of this proof, we will use the words \enquote{linear polynomial} and \enquote{quadratic polynomial} a bit more liberally than usual, in the senses of \enquote{polynomial of degree at most $1$} and \enquote{polynomial of degree at most $2$} respectively.

Now, as far as the condition $a<b$ is concerned, there are essentially three possibilities:
\begin{itemize}
\item If $\epsilon=1$ and $\epsilon'=0$, then this condition is always false, whence the inversion count for this particular pair of cases is $0$, and we are done.
\item If $\epsilon=0$ and $\epsilon'=1$, then this condition is always satisfied.
\item If $\epsilon=\epsilon'$, then this condition is equivalent to $\ell<\ell'$, i.e., to $\ell+1\leq\ell'$.
\end{itemize}
An analogous discussion for the condition $f(b)<f(a)$ shows that we are done if $\delta=0$ and $\delta'=1$ (for then the condition is always false, leading to $0$ inversions altogether), that the condition is always satisfied if $\delta=1$ and $\delta'=0$, and that in case $\delta=\delta'$, the condition is equivalent to $\ell'\leq\ell+(\zeta-\zeta')k+(\eta-\eta'-1)$. Overall, there are four possibilities for how $\ell'$ may be bounded in terms of $\ell$ in those cases we still need to investigate (and the conditions for when each possibility occurs are formulated in terms of the case parameters $\epsilon$, $\epsilon'$, $\delta$, $\delta'$ alone):
\begin{enumerate}
\item $\gamma_1k+\delta_1\leq\ell'\leq\gamma_2k+\delta_2$ (i.e., no further restrictions compared to the bounds on $\ell'$ from above).
\item $\max\{\ell+1,\gamma_1k+\delta_1\}\leq\ell'\leq\gamma_2k+\delta_2$.
\item $\gamma_1k+\delta_1\leq\ell'\leq\min\{\ell+(\zeta-\zeta')k+(\eta-\eta'-1),\gamma_2k+\delta_2\}$.
\item $\max\{\ell+1,\gamma_1k+\delta_1\}\leq\ell'\leq\min\{\ell+(\zeta-\zeta')k+(\eta-\eta'-1),\gamma_2k+\delta_2\}$.
\end{enumerate}

For bound (1), the number of $\ell'$ for each given $\ell$ does not depend on $\ell$ and is, more precisely, given by the formula
\begin{equation}\label{countingEq}
\frac{1}{4}((\gamma_2-\gamma_1)k+\delta_2-\delta_1)+1,
\end{equation}
a linear (and possibly constant, in case $\gamma_2=\gamma_1$) polynomial in $k$ (note that $\ell'$ must adhere to a specific congruence class modulo $4$, whence the division by $4$). Moreover, the number of $\ell$ is
\[
\frac{1}{4}((\alpha_2-\alpha_1)k+\beta_2-\beta_1)+1,
\]
another linear polynomial in $k$. The total number of inversions to account for is the product of these two polynomials, hence, a quadratic polynomial in $k$ that depends only on the case parameters, as required.

For bound (2), we divide the full range for $\ell$ into two subintervals (one of which may be empty), and it suffices to show that each of these two subcases accounts for a number of inversions expressed by a quadratic polynomial in $k$ (which may depend on the case parameters).
\begin{itemize}
\item First, consider those $\ell$ such that $\alpha_1k+\beta_1\leq\ell\leq\gamma_1k+\delta_1-1$. We claim that the number of integers $\ell$ satisfying these bounds as well as formula (\ref{ellEq}) can be given by a linear polynomial in $k$ (depending on the case parameters). Indeed, let $c_1\in\{0,1,2,3\}$ be the smallest nonnegative integer such that $\delta_1-1-c_1\equiv\beta_1\Mod{4}$. Due to $\ell\equiv\beta_1\Mod{4}$, the given bounds on $\ell$ are actually equivalent to $\alpha_1k+\beta_1\leq\ell\leq\gamma_1k+\delta_1-1-c_1$ (note that $\gamma_1k\equiv0\Mod{4}$ since $\gamma_1$ and $k$ are both even).
\begin{itemize}
\item If $\gamma_1<\alpha_1$, then $\gamma_1k+\delta_1-1-c_1<\alpha_1k+\beta_1$ (and, in particular, the number of $\ell$ is $0$), for otherwise,
\begin{align*}
2k&\leq(\alpha_1-\gamma_1)k\leq\delta_1-\beta_1-1-c_1\leq\delta_1-\beta_1-1=|\delta_1-\beta_1-1| \\
&\leq|\delta_1|+|\beta_1|+1\leq13,
\end{align*}
a contradiction to our assumption that $k\geq14$ (which is not yet needed at its full strength here).
\item If $\gamma_1=\alpha_1$ and $\delta_1-1-c_1<\beta_1$, then we also have $\gamma_1k+\delta_1-1-c_1<\alpha_1k+\beta_1$, and there are $0$ such $\ell$.
\item Otherwise, we have $\alpha_1k+\beta_1\leq\gamma_1k+\delta_1-1-c_1$ (a similar argumentation to the one above, using that $k\geq14$, shows that this holds if $\alpha_1<\gamma_1$), and the number of corresponding $\ell$ is given by a linear polynomial in $k$ that depends on whether or not $\gamma_1k+\delta_1-1-c_1\leq\alpha_2k+\beta_2$, which holds if and only if $\gamma_1<\alpha_2$ or $\gamma_1=\alpha_2$ and $\delta_1-1-c_1\leq\beta_2$ (note that each of these conditions is formulated in terms of case parameters alone). The said polynomial then is
\[
\frac{1}{4}((\gamma_1-\alpha_1)k+\delta_1-1-c_1-\beta_1)+1
\]
respectively
\[
\frac{1}{4}((\alpha_2-\alpha_1)k+\beta_2-\beta_1)+1.
\]
\end{itemize}
Moreover, for each such $\ell$, we have $\max\{\ell+1,\gamma_1k+\delta_1\}=\gamma_1k+\delta_1$, and so the condition on $\ell'$ simplifies to $\gamma_1k+\delta_1\leq\ell'\leq\gamma_2k+\delta_2$. Hence, for each given $\ell$, the number of matching $\ell'$ is given by a linear polynomial in $k$ (the same as in formula (\ref{countingEq})), and so in total, this subcase accounts for a number of inversions that is quadratic in $k$ (a product of two linear polynomials in $k$).
\item Now consider those $\ell$ such that $\gamma_1k+\delta_1\leq\ell\leq\alpha_2k+\beta_2$ (equivalently, the lower bound on $\ell$ may be replaced by $\gamma_1k+\delta_1-c_1+3$). If $\gamma_1>\alpha_2$, or if $\gamma_1=\alpha_2$ and $\delta_1-c_1+3>\beta_2$, then these bounds cannot be satisfied, so the number of corresponding $\ell$ is $0$. Otherwise, we have $\gamma_1k+\delta_1-c_1+3\leq\alpha_2k+\beta_2$, and the corresponding $\ell$ form an arithmetic progression with increment $4$, starting with $\max\{\alpha_1k+\beta_1,\gamma_1k+\delta_1-c_1+3\}$ and ending with $\alpha_2k+\beta_2$. For each such $\ell$, since $\max\{\ell+1,\gamma_1k+\delta_1\}=\ell+1$, we find that the corresponding $\ell'$ are characterized by the bounds $\ell+1\leq\ell'\leq\gamma_2k+\delta_2$. If $c_2\in\{0,1,2,3\}$ is the smallest nonnegative integer such that $\beta_1+1+c_2\equiv\delta_2\Mod{4}$, then for each given $\ell$, the number of matching $\ell'$ is given by the formula
\[
\max\{0,\frac{1}{4}(\gamma_2k+\delta_2-\ell-1-c_2)+1\}.
\]
The total inversion count for this range of $\ell$ is equal to one of the (possibly empty) sums
\[
\sum_{\ell=\gamma_1k+\delta_1+c_1-3,\ell\equiv\beta_1\Mod{4}}^{\gamma_2k+\delta_2-c_2+3}{\left(\frac{1}{4}(\gamma_2k+\delta_2-\ell-1-c_2)+1\right)}
\]
or
\[
\sum_{\ell=\alpha_1k+\beta_1,\ell\equiv\beta_1\Mod{4}}^{\gamma_2k+\delta_2-c_2+3}{\left(\frac{1}{4}(\gamma_2k+\delta_2-\ell-1-c_2)+1\right)},
\]
depending on whether or not $\alpha_1k+\beta_1\leq\gamma_1k+\delta_1+c_1-3$, which holds if and only if $\alpha_1<\gamma_1$, or $\alpha_1=\gamma_1$ and $\beta_1\leq\delta_1+c_1-3$. In any case, the inversion count for this range of $\ell$ is a quadratic polynomial in $k$ depending solely on the case parameters.
\end{itemize}

The treatment of bound (3) is analogous to the one for bound (2), and we omit it. The only remarkable difference is that the discussion of bound (3) requires larger lower bounds on $k$ to work than bound (2) (our assumption $k\geq14$ is strong enough, however). For example, the lower segment for $\ell$, the one where
\[
\min\{\ell+(\zeta-\zeta')k+(\eta-\eta'-1),\gamma_2k+\delta_2\}=\ell+(\zeta-\zeta')k+(\eta-\eta'-1),
\]
consists of just those $\ell$ for which
\[
\alpha_1k+\beta_1\leq\ell\leq(\gamma_2-\zeta+\zeta')k+(\delta_2-\eta+\eta'+1).
\]
Checking that these bounds are sensical (i.e., there are values of $\ell$ satisfying them as well as formula (\ref{ellEq})) in case $\alpha_1<\gamma_2-\zeta+\zeta'$ requires the following bound to be false, where $c_3\in\{0,1,2,3\}$ is the smallest nonnegative integer such that $\delta_2-\eta+\eta'+1-c_3\equiv\beta_1\Mod{4}$:
\begin{align*}
2k &\leq(\gamma_2-\zeta+\zeta'-\alpha_1)k<\beta_1-\delta_2+\eta-\eta'-1+c_3=|\beta_1-\delta_2+\eta-\eta'-1+c_3| \\
&\leq|\beta_1|+|\delta_2|+|\eta|+|\eta'|+1+c_3\leq 22.
\end{align*}

Finally, we discuss bound (4), the most complicated case. The strategy is similar to the one for bounds (2) and (3) in that we subdivide the range for $\ell$ into segments on each of which the maximum and minimum in the bounds for $\ell'$ can be simplified. Note that
\[
\max\{\ell+1,\gamma_1k+\delta_1\}=
\begin{cases}
\gamma_1k+\delta_1, & \text{if }\ell\leq\gamma_1k+\delta_1-1-c_1, \\
\ell+1, & \text{otherwise},
\end{cases}
\]
and
\begin{align*}
&\min\{\ell+(\zeta-\zeta')k+(\eta-\eta'-1),\gamma_2k+\delta_2\} \\
&=
\begin{cases}
\ell+(\zeta-\zeta')k+(\eta-\eta'-1), & \text{if }\ell\leq(\gamma_2-\zeta+\zeta')k+\delta_2-\eta'+\eta+1-c_3, \\
\gamma_2k+\delta_2, & \text{otherwise}.
\end{cases}
\end{align*}
The details of our subdivision into $\ell$-segments depend on whether or not
\begin{equation}\label{thresholdEq}
\gamma_1k+\delta_1-1-c_1\leq(\gamma_2-\zeta+\zeta')k+\delta_2-\eta'+\eta+1-c_3.
\end{equation}
\begin{itemize}
\item Subcase (a): $\gamma_1<\gamma_2-\zeta+\zeta'$, or $\gamma_1=\gamma_2-\zeta+\zeta'$ and $\delta_1-1-c_1\leq\delta_2-\eta+\eta'+1-c_3$. Then inequality (\ref{thresholdEq}) holds. We subdivide the range for $\ell$ into the following three segments:
\begin{itemize}
\item If $\alpha_1k+\beta_1\leq\ell\leq\gamma_1k+\delta_1-1-c_1$, then both the maximum and minimum simplify to the respective first option, and $\ell'$ is bounded as follows:
\[
\gamma_1k+\delta_1\leq\ell'\leq\ell+(\zeta-\zeta')k+(\eta-\eta'-1-c_4),
\]
where $c_4\in\{0,1,2,3\}$ is the smallest nonnegative integer such that $\beta_1+\eta-\eta'-1-c_4\equiv\delta_1\Mod{4}$. In particular, the number of matching $\ell'$ for each $\ell$ is given by
\[
\max\{0,\frac{1}{4}(\ell+(\zeta-\zeta'-\gamma_1)k+(\eta-\eta'-1-c_4-\delta_1))+1\}.
\]
The situation we have here is analogous to the one at the end of the discussion for bound (2), and like there, we arrive at the conclusion that the inversion count for this segment of $\ell$ is a quadratic polynomial in $k$ that only depends on the case parameters.
\item If $\gamma_1k+\delta_1+3-c_1\leq\ell\leq(\gamma_2-\zeta+\zeta')k-\eta+\eta'+1-c_3$, then $\ell'$ is bounded as follows:
\[
\ell+1+c_2\leq\ell'\leq\ell+(\zeta-\zeta')k+(\eta-\eta'-1-c_4).
\]
If $\zeta<\zeta'$, or if $\zeta=\zeta'$ and $\eta-\eta-1-c_4<1+c_2$, then these bounds are nonsensical, whence we count $0$ inversions for this segment of $\ell$. Otherwise, the $\ell'$-count per $\ell$ does not depend on $\ell$ and is equal to
\[
\frac{1}{4}((\zeta-\zeta')k+(\eta-\eta'-2-c_4-c_2))+1.
\]
The number of $\ell$ for this segment is either $0$ (which happens if and only if $\gamma_2-\zeta+\zeta'<\gamma_1$, or if $\gamma_2-\zeta+\zeta'=\gamma_1$ and $\delta_2-\eta+\eta'+1-c_3<\delta_1-c_1+3$), or it is given by one of four possible linear polynomials in $k$, depending on whether or not $\alpha_1k+\beta_1\leq\gamma_1k+\delta_1+3-c_1$ and whether or not $(\gamma_2-\zeta+\zeta')k-\eta+\eta'+1-c_3\leq\alpha_2k+\beta_2$, and either of these conditions can be characterized in terms of the case parameters alone. Therefore, the total inversion count for this range of $\ell$ is a quadratic polynomial in $k$ depending solely on the case parameters.
\item If $(\gamma_2-\zeta+\zeta')k+\delta_2-\eta+\eta'+5-c_3\leq\ell\leq\alpha_2k+\beta_2$, then the bounds on $\ell'$ simplify to $\ell+1+c_3\leq\ell'\leq\gamma_2k+\delta_2$, which means that the $\ell'$-count per $\ell$ is given by the formula
\[
\max\{0,\frac{1}{4}(\gamma_2k+\delta_2-\ell-1-c_3)+1\},
\]
and we can conclude analogously to the end of the argument for bound (2).
\end{itemize}
\item Subcase (b): $\gamma_1>\gamma_2-\zeta+\zeta'$, or $\gamma_1=\gamma_2-\zeta+\zeta'$ and $\delta_1-1-c_1>\delta_2-\eta+\eta'+1-c_3$. Then the negation of inequality (\ref{thresholdEq}) holds, and our subdivision of the range for $\ell$ into segments is as follows:
\begin{itemize}
\item If $\alpha_1k+\beta_1\leq\ell\leq(\gamma_2-\zeta+\zeta')k+\delta_2-\eta'+\eta+1-c_3$, then the bounds on $\ell'$, and thus the $\ell'$-count per $\ell$, is the same as for the first segment in Subcase (a). The situation is again analogous to the one at the end of the argument for bound (2) and can be concluded as such.
\item If $(\gamma_2-\zeta+\zeta')k+\delta_2-\eta'+\eta-c_3+5\leq\ell\leq\gamma_1k+\delta_1-1-c_1$, then the bounds on $\ell'$ simplify to the \enquote{trivial} ones, $\gamma_1k+\delta_1\leq\ell'\leq\gamma_2k+\delta_2$, whence the $\ell'$-count per $\ell$ is constant at
\[
\frac{1}{4}((\gamma_2-\gamma_1)k+(\delta_2-\delta_1))+1,
\]
and the count for $\ell$ is either $0$ (namely, if and only if either $\gamma_1<\gamma_2-\zeta+\zeta'$, or $\gamma_1=\gamma_2-\zeta+\zeta'$ and $\delta_1-1-c_1<\delta_2-\eta'+\eta-c_3+5$), or is given by a linear polynomial in $k$ that depends solely on the case parameters (as for the second segment in Subcase (a), there are four different formulas for this polynomial, depending on whether or not $\alpha_1k+\beta_1\leq(\gamma_2-\zeta+\zeta')k+\delta_2-\eta'+\eta-c_3+5$ and whether or not $\gamma_1k+\delta_1-1-c_1\leq\alpha_2k+\beta_2$). We remark that it is in the argumentation for this segment that the full power of the assumption $k\geq14$ is needed. Indeed, in order to show that $\gamma_1>\gamma_2-\zeta+\zeta'$ implies that $\gamma_1k+\delta_1-1-c_1>(\gamma_2-\zeta+\zeta')k+\delta_2-\eta'+\eta-c_3+5$ (a part of the above \enquote{if and only if} characterizing when the $\ell$-count is $0$), one argues that otherwise,
\begin{align*}
2k &\leq(\gamma_1-\gamma_2+\zeta-\zeta')k\leq\delta_2-\eta+\eta'-c_3+5-\delta_1+1+c_1 \\
&\leq|\delta_2|+|\eta|+|\eta'|+5+|\delta_1|+1+3\leq 6+3+3+5+6+1+3=27,
\end{align*}
a contradiction if $k\geq14$.
\item If $\gamma_1k+\delta_1+3-c_1\leq\ell\leq\alpha_2k+\beta_2$, then the $\ell'$-bounds are the same as for the third $\ell$-segment in Subcase (a), whence the $\ell'$-count per $\ell$ is the same as there, and again, we can conclude analogously to the end of the argument for bound (2).
\end{itemize}
\end{itemize}
\end{proof}

\section{Appendix D: Noncyclic groups of order \texorpdfstring{$16$}{16} have complete mappings of both parities}

We will prove the statement from the title of this Appendix. The Small Groups Library of GAP lists the groups of order $16$ as $\SmallGroup(16,i)$ with $i\in\{1,2,\ldots,14\}$. Note that $C_{16}=\SmallGroup(16,1)$, so we only need to consider indices $i$ with $2\leq i\leq14$. A few cases are easy to deal with using what was said in other parts of this paper:
\begin{itemize}
\item Any Singer cycle of $C_2^4=\SmallGroup(16,14)$ is an even complete mapping of it. An odd complete mapping of this group was given in the proof of Theorem \ref{mainTheo1} in Section \ref{sec2}.
\item $\AC_{16}=C_8\times C_2=\SmallGroup(16,5)$ has complete mappings of both parities by \cite[Theorem 6.6]{BGHJ91a} and \cite[Theorem 7]{FGM78a} -- see our proof of Proposition \ref{specialCasesProp}(1) in Section \ref{sec4}.
\item $\D_{16}=\SmallGroup(16,7)$ has complete mappings of both parities by \cite[Theorem 5.8]{BGHJ91a} and \cite[proof of Lemma 1]{HP55a} -- see the proof of Proposition \ref{specialCasesProp}(2).
\item $\Q_{16}=\SmallGroup(16,9)$ has complete mappings of both parities by \cite[Theorem 1]{WL00a} and \cite[proof of Lemma 1]{HP55a} -- see the proof of Proposition \ref{specialCasesProp}(3).
\item $\SD_{16}=\SmallGroup(16,8)$ has complete mappings of both parities by \cite[proof of Lemma 1]{HP55a} and our Appendix A -- see the proof of Proposition \ref{specialCasesProp}(4).
\item $\M_{16}=\SmallGroup(16,6)$ has an odd complete mapping by our Appendix B. An even complete mapping of it will be specified below.
\end{itemize}

It remains to specify an even complete mapping of $\SmallGroup(16,6)$, and complete mappings of both parities of $\SmallGroup(16,i)$ for $i\in\{2,3,4,10,11,12,13\}$. We found such complete mappings with a simple random search algorithm that we implemented in GAP \cite{GAP4}. We will give a refined polycyclic presentation of each group $G$ in question (read off from GAP). For example, for $G=\SmallGroup(16,6)=\M_{16}$, that presentation is
\begin{align*}
\langle x,y,z,t: &y^2=t^2=[z,x]=[t,x]=[z,y]=[t,y]=[t,z]=1, \\
&x^2=z, z^2=t, [y,x]=t\rangle.
\end{align*}
The refined polycyclic presentation leads to a normal form representation of the elements of $G$ as products of the generators from the presentation. These elements are listed internally in GAP in increasing lexicographical order. That is, if the four generators from the presentation are denoted by $x$, $y$, $z$ and $t$, the ordering of the group elements is as follows:
\begin{center}
\begin{longtable}[H]{|c|c|c|c|c|c|c|c|c|c|c|c|c|c|c|c|c|}
\hline
$i$ & 1 & 2 & 3 & 4 & 5 & 6 & 7 & 8 & 9 & 10 & 11 & 12 & 13 & 14 & 15 & 16 \\ \hline
$g_i$ & $1$ & $x$ & $y$ & $z$ & $t$ & $xy$ & $xz$ & $xt$ & $yz$ & $yt$ & $zt$ & $xyz$ & $xyt$ & $xzt$ & $yzt$ & $xyzt$ \\ \hline
\caption{Lexicographic ordering of group elements}
\end{longtable}
\end{center}
We specify the complete mappings themselves as permutations of the number set $\{1,2,\ldots,16\}$, identifying each number $i$ with the corresponding element $g_i\in G$. For example, for $G=\SmallGroup(16,6)$, the even complete mapping we found through random search is specified as
\[
(1,7,8,4,6,13,3,2,12,14,9,10,5,15,16);
\]
note that this is a $15$-cycle -- fixed points are omitted by convention. For the remaining groups, where we need to find complete mappings of both parities, our computations gave the following results:
\begin{itemize}
\item $G=\SmallGroup(16,2)=C_4\times C_4$:
\begin{itemize}
\item a refined polycyclic presentation of $G$:
\begin{align*}
\langle x,y,z,t: &z^2=t^2=[y,x]=[z,x]=[t,x]=[z,y]=[t,y]=[t,z]=1, \\
&x^2=z, y^2=t\rangle.
\end{align*}
\item an even complete mapping of $G$:
\[
(1,4,6,2,16,9,5)(3,10,7,11,8)(13,14,15).
\]
\item an odd complete mapping of $G$:
\[
(1,12,15,4,14,9,8,16,3,11)(2,7,5,13,6).
\]
\end{itemize}
\item $G=\SmallGroup(16,3)=C_2\ltimes (C_4\times C_2)$:
\begin{itemize}
\item a refined polycyclic presentation of $G$:
\begin{align*}
\langle x,y,z,t: &y^2=z^2=t^2=[z,x]=[t,x]=[z,y]=[t,y]=[t,z]=1, \\
&x^2=t,[y,x]=z\rangle.
\end{align*}
\item an even complete mapping of $G$:
\[
(2,15,12,11,7,3,9,4,5,13,14,6)(10,16).
\]
\item an odd complete mapping of $G$:
\[
(1,6,9,16,12,8,13,4,10,7,3,15,5,2).
\]
\end{itemize}
\item $G=\SmallGroup(16,4)=C_4\ltimes C_4$:
\begin{itemize}
\item a refined polycyclic presentation of $G$:
\begin{align*}
\langle x,y,z,t: &z^2=t^2=[z,x]=[t,x]=[z,y]=[t,y]=[t,z]=1, \\
&x^2=t, y^2=z, [y,x]=z\rangle.
\end{align*}
\item an even complete mapping of $G$:
\[
(1,8,10,12,13,4,11,3,9,15,2)(5,14,6,7,16).
\]
\item an odd complete mapping of $G$:
\[
(1,2,11,8,15,12,7,14,3,5,13,16,4,9).
\]
\end{itemize}
\item $G=\SmallGroup(16,10)=C_4\times C_2\times C_2$:
\begin{itemize}
\item a refined polycyclic presentation of $G$:
\begin{align*}
\langle x,y,z,t: &y^2=z^2=t^2=[y,x]=[z,x]=[t,x]=[z,y]=[t,y]=[t,z]=1, \\
&x^2=t\rangle.
\end{align*}
\item an even complete mapping of $G$:
\[
(1,12,4,5,7,6,13,8)(3,14,9)(10,11,15,16).
\]
\item an odd complete mapping of $G$:
\[
(1,13,16,10,8,15,12,14,2)(3,5,7,6,11,4).
\]
\end{itemize}
\item $G=\SmallGroup(16,11)=C_2\times \D_8$:
\begin{itemize}
\item a refined polycyclic presentation of $G$:
\begin{align*}
\langle x,y,z,t: &x^2=y^2=z^2=t^2=[z,x]=[t,x]=[z,y]=[t,y]=[t,z]=1, \\
&[y,x]=t\rangle.
\end{align*}
\item an even complete mapping of $G$:
\[
(1,6,10,8,14,7,13,11)(2,15,4,9,16)(3,12).
\]
\item an odd complete mapping of $G$:
\[
(1,13,3,14,12,2,8,11,16)(4,7,6,15,10,5).
\]
\end{itemize}
\item $G=\SmallGroup(16,12)=C_2\times\Q_8$:
\begin{itemize}
\item a refined polycyclic presentation of $G$:
\begin{align*}
\langle x,y,z,t: &z^2=t^2=[z,x]=[t,x]=[z,y]=[t,y]=[t,z]=1, \\
&x^2=[y,x]=t\rangle.
\end{align*}
\item an even complete mapping of $G$:
\[
(1,16,9)(2,11,3,12,7,14)(4,6,5,13,10,15).
\]
\item an odd complete mapping of $G$:
\[
(1,13,5,14,8)(3,7)(4,15,11,10)(6,9,12,16).
\]
\end{itemize}
\item $G=\SmallGroup(16,13)=C_2\ltimes(C_4\times C_2)=\D_8\circ C_4$:
\begin{itemize}
\item a refined polycyclic presentation of $G$:
\begin{align*}
\langle x,y,z,t: &x^2=y^2=t^2=[z,x]=[t,x]=[z,y]=[t,y]=[t,z]=1, \\
&z^2=[y,x]=t\rangle.
\end{align*}
\item an even complete mapping of $G$:
\[
(1,11,4,16,8,13,9,3,10,7,12,5,2)(6,14,15).
\]
\item an odd complete mapping of $G$:
\[
(1,15,7,2,3,10,6,8,16)(4,5,14,11,13,9).
\]
\end{itemize}
\end{itemize}
\end{document}